\documentclass[12pt]{amsart}
\usepackage{amscd,amssymb,url,graphicx,color}
\usepackage[all]{xy}

\newtheorem{thm}{Theorem}
\newtheorem{cor}[thm]{Corollary}
\newtheorem{lemma}[thm]{Lemma}
\newtheorem{prop}[thm]{Proposition}
\theoremstyle{definition}

\newtheorem{ex}[thm]{Example}
\newtheorem{remark}[thm]{Remark}
\newcommand{\sections}{\renewcommand{\thethm}{\thesection.\arabic{thm}}
           \setcounter{thm}{0}}
\newcommand{\nosubsections}{\renewcommand{\thethm}{\thesection.\arabic{thm}}
           \setcounter{thm}{0}}

\newcommand{\co}{\colon}

\newcommand{\bbF}{\mathbb{F}}
\newcommand{\bbH}{\mathbb{H}}
\newcommand{\bbQ}{\mathbb{Q}}
\newcommand{\bbR}{\mathbb{R}}
\newcommand{\bbZ}{\mathbb{Z}}
\newcommand{\cA}{\mathcal{A}}
\newcommand{\cH}{\mathcal{H}}
\newcommand{\cS}{\mathcal{S}}
\newcommand{\za}{\alpha}
\newcommand{\zd}{\delta}
\newcommand{\ze}{\epsilon}
\newcommand{\zf}{\phi}
\newcommand{\zg}{\gamma}
\newcommand{\zh}{\eta}
\newcommand{\zj}{\psi}
\newcommand{\zl}{\lambda}
\newcommand{\zm}{\mu}
\newcommand{\zp}{\pi}
\newcommand{\zq}{\theta}
\newcommand{\zs}{\sigma}
\newcommand{\zt}{\tau}
\newcommand{\zv}{\varphi}
\newcommand{\zw}{\omega}
\newcommand{\zD}{\Delta}
\newcommand{\zF}{\Phi}
\newcommand{\zG}{\Gamma}
\newcommand{\zJ}{\Psi}
\newcommand{\zL}{\Lambda}
\newcommand{\zQ}{\Theta}

\begin{document}
\title[NET maps]{Modular groups, Hurwitz classes and\\ 
dynamic portraits of NET maps}

\begin{author}{William Floyd}
\address{Department of Mathematics\\ Virginia Tech\\
Blacksburg, VA 24061\\ USA}
\email{floyd@math.vt.edu}
\urladdr{http://www.math.vt.edu/people/floyd}
\end{author}

\begin{author}{Walter Parry}
\email{walter.parry@emich.edu}
\end{author}

\begin{author}{Kevin M. Pilgrim}
\address{Department of Mathematics, Indiana University, Bloomington, 
IN 47405, USA}
\email{pilgrim@indiana.edu}
\end{author}

\date{\today}

\begin{abstract} An orientation-preserving branched covering $f: S^2
\to S^2$ is a \emph{nearly Euclidean Thurston} (NET) map if each
critical point is simple and its postcritical set has exactly four
points.  Inspired by classical, non-dynamical notions such as Hurwitz
equivalence of branched covers of surfaces, we develop invariants for
such maps.  We then apply these notions to the classification and
enumeration of NET maps. As an application, we obtain a complete
classification of the dynamic critical orbit portraits of NET maps.
\end{abstract}

\subjclass[2010]{Primary: 36F10; Secondary: 57M12}

\keywords{Thurston map, virtual endomorphism, branched covering}

\thanks{}
\maketitle

\tableofcontents 

\sections

\section{Introduction }\label{sec:intro}\nosubsections

This paper is part of our program to  investigate nearly
Euclidean Thurston (NET) maps.  A Thurston map is NET if each critical
point is simple and it has exactly four postcritical points. 

The program began with \cite{cfpp}, written with Jim Cannon.  There,
we showed that each NET map $f$ admits a factorization $f=h\circ g$
where $g$ is a Euclidean Thurston map which is affine in suitable
coordinates and $h$ is a homeomorphism. This property motivated the
term nearly Euclidean. Thus NET maps are among the simplest Thurston
maps, and their associated Teichm\"uller spaces have complex dimension
1.  In \cite{cfpp}, we exploited the connection to affine geometry to
give effective algorithms for the computation of fundamental dynamic
invariants, such as the pullback function on homotopy classes of
simple closed curves.

Next, \cite{fpp1} shows that every NET map has a presentation which
allows it to be described by a simple diagram, such as that shown in
Figures \ref{fig:PrenDgm}, \ref{fig:firstprendgm} and
\ref{fig:degfourprendgm}, and gives examples showing how to find such
diagrams in concrete cases.  The data in the diagram can then be
translated into input for a computer program, {\tt NETmap} (see \cite{NETmap}), which
implements the algorithms to compute the invariants.  The NET map
website \cite{NET} contains an overview, papers, many examples, and
executable files for the computer program {\tt NETmap}.

The survey \cite{fkklpps}, written jointly with G. Kelsey, S. Koch,
R. Lodge, and E. Saenz, reports on findings and new phenomena from a
study of the data generated and on connections to other areas.

This paper develops other invariants of NET maps, focusing on
non-dynamical invariants.

We now list these invariants and outline our first main results, in
roughly decreasing order of sensitivity.  Suppose $f, g$ are Thurston
maps with postcritical sets $P_f, P_g$.  In the first four items below, $\phi, \psi$ are orientation-preserving homeomorphisms of the sphere to itself. 

\begin{itemize} 
  \item We recall that $f, g$ are \emph{Thurston equivalent} if
there exist such $\phi, \psi$, with $\phi(P_f)=\psi(P_f)=P_g$, $g \circ \psi =
\phi \circ f$, $\phi|_{P_f}=\psi|_{P_f}$, and $\phi, \psi$ are isotopic relative to $P_f$.  Put
another way, $f$ and $g$ are conjugate up to isotopy relative to their
postcritical sets.  \vskip 0.15in
  \item We say $f, g$ are \emph{Hurwitz equivalent for the pure
modular group} if there exist such $\phi, \psi$ with $\phi(P_f)=\psi(P_f)=P_g$, $g \circ \psi =
\phi \circ f$, and  $\phi|_{P_f}=\psi|_{P_f}$.  The motivation for
the terminology: pre- and/or post-composing $f$ with a homeomorphism
fixing $P_f$ pointwise yields a map $g$ which is equivalent to $f$. In
suitable Euclidean coordinates, the pure modular group is
$\text{PMod}(S^2, P_f) = \overline{\Gamma}(2) < PSL(2,\mathbb{Z})$.
\vskip0.15in
  \item We say $f, g$ are \emph{Hurwitz equivalent for the modular
group} if there exist such $\phi, \psi$ with $\phi(P_f)=\psi(P_f)=P_g$ and $g \circ \psi = \phi
\circ f$.  Theorem \ref{thm:hurwitz} gives a complete algebraic
invariant of the modular group Hurwitz type.  In suitable Euclidean
coordinates, the modular group is $\text{Mod}(S^2, P_f) =
PSL(2,\mathbb{Z}) \ltimes (\mathbb{Z}/2\mathbb{Z})^2$.  \vskip 0.15in
\item We say $f, g$ are \emph{topologically equivalent} if
there exist such $\phi, \psi$ with $g \circ \psi = \phi \circ f$.  
%\wrp{\item Two branched coverings $f,g\co X\to Y$ between closed surfaces
%are \textbf{topologically equivalent} if there exist homeomorphisms
%$\zj\co X\to X$ and $\zf\co Y\to Y$  with $g \circ \psi = \phi \circ f$.} 
\vskip 0.15in
\item Every NET map $f$ is topologically equivalent to a Euclidean NET
map $g$ induced by a linear map $x \mapsto Ax$, where $A$ is a
diagonal matrix with entries $(m,n)$ and $n|m$. The pair $(m,n)$ are
the \emph{elementary divisors} of $f$. Theorem \ref{thm:eleydivrs}
shows that the elementary divisors form a complete invariant of the
topological equivalence class of a NET map; see also Theorem
\ref{thm:hurwitz_special}.  Hence all NET maps in a given modular
group Hurwitz class have the same elementary divisors.  \vskip 0.15in

\item Most NET maps have four critical values. There are two
\emph{atypical} cases: elementary divisors $(2,1)$ and $(2,2)$. Maps
with these elementary divisors have two and three critical values,
respectively.  In atypical cases, modular group Hurwitz classes may
contain both NET and non-NET maps: the postcritical sets of some maps
may have fewer than four points. In \emph{typical} cases, that is, all
other cases, modular group Hurwitz classes consist entirely of NET
maps.  \vskip 0.15in

\item Given $y \in S^2$,  the collection
$\{ \deg(f,x) : f(x)=y\}$ of local degrees defines a partition of
$\deg(f)$. The \emph{branch data} associated to $f$ is the set of such
partitions as $y$ varies in the set of branch values of $f$.  Section
\ref{sec:bd} gives a classification of the branch data associated to
NET maps.  There are three basic types. In even degrees, in addition
to the obvious Riemann-Hurwitz condition, there is an
\emph{exceptional condition}: for branch data of type 3, the degree
must be divisible by 4 \cite[Thm. 3.8]{pp}.
\end{itemize}

\subsection*{The translation subgroup} The modular group of the
four-times marked sphere $(S^2,P_f)$ contains a distinguished subgroup
isomorphic to $(\mathbb{Z}/2\mathbb{Z})^2$.  In suitable Euclidean
coordinates, its elements are induced by translations.  This group
acts trivially on the Teichm\"uller space modelled on $(S^2, P_f)$.
Suppose $h: (S^2, P_f) \to (S^2, P_f)$ represents an element of this
group and $g=f\circ h$. Then many fundamental invariants of $f$ and
$g$ coincide, though their dynamic portraits may be different.  It
may also happen that $g=f$, i.e. $h$ is a deck transformation of the
covering induced by $f$.  This creates some subtleties; \S 3 gives a
thorough analysis.

\subsection*{Applications}  

Our first application is to the enumeration of representatives of
NET map modular group Hurwitz classes.  The website \cite{NET}
contains data organized first by degree, then by elementary divisors,
and finally by modular group Hurwitz class, giving one
representative for every modular group Hurwitz class through degree
30. Table~\ref{tab:summary} gives the numbers of modular group
Hurwitz classes of NET maps through degree 9.  For every choice
$(m,n)$ of elementary divisors, one modular group Hurwitz class of NET
maps consists of all Euclidean NET maps with elementary divisors
$(m,n)$.

\begin{table}
\begin{center}
\begin{tabular}{r|ccccccccccc}
$d$  &  2  &  3  &  \multicolumn{2}{c}{4}  &  5  &  6  &  7  &  
\multicolumn{2}{c}{8}  &  \multicolumn{2}{c}{9}\\ \hline
$(m,n)$  & $(2,1)$ & $(3,1)$ & $(4,1)$ & $(2,2)$ & $(5,1)$ & $(6,1)$
& $(7,1)$ & $(8,1)$ & $(4,2)$ & $(9,1)$ & $(3,3)$  \\  \hline
$N$ &  3   &  9  &  24 & 8  &  25 & 88 & 47 & 133 & 85 & 120 & 43\\ 
\end{tabular}
\end{center}
\caption{The number $N$ of modular group Hurwitz classes of NET maps 
with elementary divisors $(m,n)$ through degree $d=9$}
\label{tab:summary} \end{table}

We now turn to dynamical applications. Suppose $f$ is a Thurston map
with set of critical points $C_f$ and postcritical set $P_f$. The most
rudimentary dynamical invariant of $f$ is its \emph{dynamic portrait}.
This invariant records the restriction $f: C_f \cup P_f \to P_f$ and
the local degree $\deg(f,x)$ for $x \in C_f \cup P_f$.  The
website \cite{NET} contains data organized by degree and dynamic
portrait, giving one representative for every dynamic portrait through
degree 40.  Table \ref{tab:portraitnum} gives a precise count of the
number of such NET map dynamic portraits in each degree.  Theorem
\ref{thm:portrait} characterizes when a given pair consisting of
elementary divisors $(m,n)$ and a dynamic portrait $\Gamma$ is
realizable by a NET map. We also present algorithms for computing the
correspondence between dynamic portraits of NET maps and NET map
presentations.  We describe these algorithms in
Sections~\ref{sec:maptopic} and \ref{sec:pictomap}.

As an example of how non-dynamical invariants influence dynamical
properties, we present the following.  Below, a Thurston map $f$ is
\emph{B\"ottcher expanding} if away from periodic critical points,
it is uniformly expanding with respect to a natural orbifold metric on
the orbifold associated to $f$.  Any rational Thurston map is
B\"ottcher expanding.

\begin{thm} \label{thm:consists_of_expanding} Suppose $f$ is a
non-Euclidean NET map with elementary divisors $(m,n)$ and
postcritical set $P_f$. If $n>1$ then $f$ is isotopic relative to $P_f$ to
a B\"ottcher expanding map.
\end{thm}

The assumption that $f$ is not Euclidean is necessary: if $f$ is the
map induced by the matrix $A=2 \cdot \left[\begin{smallmatrix}2 & 1 \\
1 & 1\end{smallmatrix}\right]$, then $n=2$ while $A$ has a positive
real eigenvalue which is strictly less than one; this is an obstruction
to expansion in the isotopy class. There are plenty of B\"ottcher
expanding NET maps with $n=1$; among them are many rational maps too,
e.g. the Douady rabbit polynomial. So the condition is far from sharp.

\begin{proof} Recall that $f$ is \emph{Levy-free} if there is no
simple closed curve $\zg\subseteq S^2-P_f$ which is essential, not
peripheral and having a lift $\zd$ homotopic to $\zg$ which $f$ maps
to $\zg$ with degree 1.  Using Theorem 4.1 of \cite{cfpp}, we can see
that the degree with which $f$ maps $\zd$ to $\zg$ is divisible by
$n>1$.  So $f$ is Levy-free. By \cite[Theorem C]{bd0}, $f$ is
isotopic to a B\"ottcher expanding map.
\end{proof}

If a modular group Hurwitz class $\cH$ contains a non-Euclidean NET
map, then every map in $\cH$ is non-Euclidean and we call $\cH$
\emph{non-Euclidean}. Any two maps in $\cH$ have the same elementary
divisors.

\begin{cor} \label{cor:consists_of_expanding} Suppose $\cH$ is a
non-Euclidean modular group Hurwitz class of NET maps and the
elementary divisor $n$ is strictly larger than one. Then every isotopy
class in $\cH$ contains a B\"ottcher expanding representative.
\end{cor}

As another application, we give an effective algorithm for calculating
a basic dynamical invariant.  Suppose $f: (S^2, P_f)
\to (S^2, P_f)$ is a NET map with postcritical set $P_f$. If a
homeomorphism $\psi$ represents an element of $\text{PMod}(S^2, P_f)$
and lifts under $f$ to a map $\widetilde{\psi}$ which again represents
an element of $\text{PMod}(S^2, P_f)$, the assignment $\psi \mapsto
\widetilde{\psi}$ induces the \emph{virtual endomorphism on the pure
mapping class group} associated to $f$.  

For generic Thurston maps,
this is difficult to compute, even for maps with $\#P_f=4$.  One method is the following. As shown by S. Koch \cite{Ko}, a Thurston map $f$ has an associated algebraic correspondence $X, Y: \mathcal{W} \to \mathcal{M}$ on the moduli space $\mathcal{M}$ of conformal configurations of embeddings of $P_f$ into the Riemann sphere.  Here, $Y$ is a finite unramified covering map, and $X$ is holomorphic. If $\#P_f=4$, in suitable coordinates, $\mathcal{M}=\mathbb{C}-\{0,1\}$, $\text{PMod}(S^2, P_f)$ is identified with the fundamental group of $\mathcal{M}$, $\mathcal{W}$ is a compact Riemann surface minus a finite set of points and is given as a locus $\mathcal{W}:=\{P(x,y)=0 : x, y \in \mathcal{M}\}\subset \mathbb{C}^2$, and the maps $X, Y$ are the coordinate projections.   
In the special case when $f$ is rational (for ease of exposition here), there is a distinguished ``fixed'' point $w_0 \in \mathcal{W}$ with $m_0:=X(w_0)=Y(w_0)$, and the virtual endomorphism is then the induced map on fundamental group $X_*\circ Y_*^{-1}$.  In cases with few critical points, such as those considered in \cite{BN} and \cite{L}, the defining equation $P(x,y)=0$ for the correspondence can be found explicitly, and from this the virtual endomorphism may be computed.  For maps with many critical points, this quickly becomes computationally intractable.  

We consider a slightly more general notion, and calculate it using different methods.  We relax the condition that the homeomorphisms fix $P$ pointwise to merely fixing $P$ setwise, and we obtain an associated
virtual multi-endomorphism: the lift $\widetilde{\psi}$ might not be
well-defined, due to the possible presence of deck transformations.
Theorem \ref{thm:viredmp} gives a method for calculating $\psi
\dashrightarrow \widetilde{\psi}$ from knowledge of the \emph{slope
function} $\mu_f: \overline{\mathbb{Q}}
\dasharrow \overline{\mathbb{Q}}$; see \S \ref{sec:background}.  This method is then
implemented in the {\tt NETmap} program.  Our method relies on two
ingredients.  First, there is a simple finite linear-algebraic
condition for an element $\psi\in\text{Mod}(S^2, P_f)$ to be liftable
under $f$, i.e. to lie in the domain of the virtual
multi-endomorphism (Lemma \ref{lemma:when_liftable}). To
calculate its image $\widetilde{\psi}$, we exploit two facts: (i) the
linear part of an element of $\text{Mod}(S^2, P_f)$ is uniquely
determined by its value at a pair of distinct extended rationals $p/q,
r/s \in \mathbb{Q}\cup \{1/0\}=:\overline{\mathbb{Q}}$, and (ii) the
slope function $\mu_f$ is algorithmically computable
\cite{cfpp}.  Section \ref{sec:exvme} illustrates this in a
complicated example.

As a further application, we obtain information about the geometry of the correspondence $X, Y: \mathcal{W} \to \mathcal{M}$.  The {\tt
NETmap} program uses classical geometric methods to find explicit matrix generators
for the domain of the virtual endomorphism (equivalently, for the Fuchsian group that uniformizes $\mathcal{W}$), and their images under the virtual endomorphism.  The lifting of complex structures under $f$ defines 
an analytic self-map $\sigma_f: \mathbb{H} \to
\mathbb{H}$ on the Teichm\"uller space modelled on $(S^2, P_f)$ that covers the correspondence $X\circ Y^{-1}$. The {\tt NETmap} program translates algebraic information about the virtual endomorphism into a rather complete geometric 
description about $\sigma_f$; cf. \cite{cfpp}.  The data on the website tabulates this geometric information.

\subsection*{Outline} Section 2 collects concepts and facts used
throughout the paper. Sections 3 and 4 discuss various modular groups
and related actions. Sections 5 and 6 introduce Hurwitz classes,
Hurwitz invariants and elementary divisors. Section 7 illustrates the calculation of the virtual endomorphism in a concrete complicated example. Section 8 introduces
branch data, used in section 9 to classify dynamic portraits. Sections
10 and 11 give algorithms for constructing a portrait from a map and
for constructing maps with a given portrait, respectively.

\subsection*{Acknowledgements}
The authors gratefully acknowledge support from the American Institute
for Mathematics.  Kevin Pilgrim was also supported by Simons grant \#245269.

\section{Background}\label{sec:background}\nosubsections In this
section we summarize some background material from \cite{cfpp} which
will be used repeatedly.  Let $f\co S^2\to S^2$ be a NET map.  There
exist a lattice $\zL_i$ in $\bbR^2$ and branched covering maps $q_i\co
\bbR^2\to \bbR^2/2\zL_i$, $p_i\co \bbR^2/2\zL_i\to S^2$ and
$\zp_i=p_i\circ q_i$ for $i\in \{1,2\}$ so that $f\circ \zp_1=\zp_2$,
as in Figure 1 in Section 1 of \cite{cfpp}.  We may always, and
usually do, assume $\Lambda_2 = \mathbb{Z}^2$.  The postcritical set
of $f$ is $P_2=\zp_2(\zL_2)$.  Statement 2 of Lemma 1.3 of \cite{cfpp}
states that $f^{-1}(P_2)$ contains exactly four points which are not
critical points of $f$.  This set of four points is
$P_1=\zp_1(\zL_1)$.  The covering $\zp_i$ is normal with group of deck
transformations equal to the group $\zG_i$ of all Euclidean isometries
of the form $x\mapsto 2\zl\pm x$ for $\zl\in \zL_i$ and $i\in
\{1,2\}$.  We will often use the finite Abelian group
$\cA:=\zL_2/2\zL_1$.

We recall Lemmas 2.1 and 2.2 stated in \cite{fpp1}.

\begin{lemma}\label{lemma:inducedown} Let $\zL$ and $\zL'$ be lattices
in $\mathbb{R}^2$.  Let $\zG$, respectively $\zG'$, be the groups of
Euclidean isometries of the form $x\mapsto 2\zl\pm x$ for some $\zl\in
\zL$, respectively $\zl\in \zL'$.  Also let $\zp\co \mathbb{R}^2\to
\mathbb{R}^2/\zG$ and $\zp'\co \mathbb{R}^2\to \mathbb{R}^2/\zG'$ be
the canonical quotient maps.  Let $\zF\co \mathbb{R}^2\to
\mathbb{R}^2$ be an affine isomorphism such that $\zF(\zL)\subseteq
\zL'$.  Then $\zF$ induces a branched covering map $\zf\co
\mathbb{R}^2/\zG\to \mathbb{R}^2/\zG'$ such that $\zf\circ
\zp=\zp'\circ \zF$.  The map $\zF$ preserves orientation if and only
if $\zf$ preserves orientation.  The set $\zF(\zL)$ is a coset of a
sublattice $\zL''$ of $\zL'$, and the degree of $\zf$ equals the index
$[\zL':\zL'']$.
\end{lemma}

\begin{lemma}\label{lemma:induceup} Let $\zL$ and $\zL'$ be lattices
in $\mathbb{R}^2$.  Let $\zG$, respectively $\zG'$, be the group of
Euclidean isometries of the form $x\mapsto 2\zl\pm x$ for some $\zl\in
\zL$, respectively $\zL'$.  Also let $\zp\co \mathbb{R}^2\to
\mathbb{R}^2/\zG$ and $\zp'\co \mathbb{R}^2\to \mathbb{R}^2/\zG'$ be
the canonical quotient maps.  Let $\zf\co \mathbb{R}^2/\zG\to
\mathbb{R}^2/\zG'$ be a branched covering map such that
$\zf(\zp(\zL))\subseteq \zp'(\zL')$.
Then we have the following three statements.
\begin{enumerate}
  \item There exists a homeomorphism $\zF\co \mathbb{R}^2\to
\mathbb{R}^2$ such that the restriction of $\zF$ to $\zL$ is affine
and $\zf\circ \zp=\zp'\circ\zF$.  If $\zp'(0)\in \zf(\zp(\zL))$, then
$\zF(\zL)$ is a sublattice of $\zL'$.
  \item There exists an affine isomorphism $\zJ\co \mathbb{R}^2\to
\mathbb{R}^2$ such that the branched map of
Lemma~\ref{lemma:inducedown} which $\zJ$ induces from
$\mathbb{R}^2/\zG$ to $\mathbb{R}^2/\zG'$ is $\zf$ up to isotopy rel
$\zp(\zL)$.  If $\zp'(0)\in \zf(\zp(\zL))$, then $\zJ(\zL)$ is a
sublattice of $\zL'$.
  \item The maps $\zF$ and $\zJ$ are unique up to precomposing with an
element of $\zG$.  They are also unique up to postcomposing with an
element of $\zG'$.
\end{enumerate}
\end{lemma}

We conclude this section by quickly recalling some dynamical
invariants defined by pullback.

Suppose $f$ is a NET map with postcritical set $P_f$.  The set of
(unoriented, simple, essential, non-peripheral, closed, up to free
homotopy) curves on $S^2-P_f$ is naturally identified with
$\overline{\mathbb{Q}}:=\mathbb{Q}\cup \{\frac{1}{0}\}$ via slopes of
representing geodesics with respect to the Euclidean metric induced
via pushforward under $\pi_2$. Pulling back via $f$ induces the
\emph{slope function} $\mu_f: \overline{\mathbb{Q}} \to
\overline{\mathbb{Q}}\cup\{\odot\}$,  where
$\odot$ is a symbol representing the set of inessential and peripheral
curves. \cite[\S 5]{cfpp} gives an algorithm for calculating $\mu_f$,
implemented in the program {\tt NETmap}.

Let $\zg$ be a simple closed curve in $S^2-P_f$ with slope $s\in
\overline{\bbQ}$.  The \emph{multiplier} of both $\zg$ and $s$ is defined as a fraction
$\frac{c}{d}$.  The numerator $c$ is the number of connected
components of $f^{-1}(\zg)$ which are neither inessential nor
peripheral in $S^2-P_f$.  Theorem 4.1 of \cite{cfpp} shows that the
local degrees of $f$ on these connected components are equal.  This
common local degree is $d$.

The set of marked complex structures on $(S^2, P_f)$ is naturally
identified with the upper half-plane $\mathbb{H} \subset \mathbb{C}$.
We briefly recall this identification. To $\tau \in \mathbb{H}$, one
associates the lattice $\Lambda_\tau=\mathbb{Z}+\tau\mathbb{Z}$, the
group $\Gamma_\tau:=\{x \mapsto 2\lambda \pm x, \lambda \in
\Lambda_\tau\}$, and the quotient $X_\tau:=\mathbb{C}/\Gamma_\tau$
marked at the four points $\Lambda_\tau/\Gamma_\tau$ and equipped with
special generators for its orbifold fundamental group in the obvious
way. Conversely, such a quotient equipped with such generators yields
a lattice in $\mathbb{C}$ equipped with an ordered basis $(\omega_1,
\omega_2)$ with $\Im(\omega_2/\omega_1)>0$, and this yields an element
$\tau:=\omega_2/\omega_1 \in \mathbb{H}$.  

Pulling back complex structures on $(S^2, P_f)$ under $f$ induces an
analytic self-map $\sigma_f: \mathbb{H} \to \mathbb{H}$.  The map
$\sigma_f$ extends continuously to the Weil-Petersson boundary
$\overline{\mathbb{Q}}$ so that $\mu_f(p/q)=p'/q'$ if and only if
$\sigma_f(-q/p)=-q'/p'$.

\section{Modular group actions}\label{sec:mdrgp}\nosubsections

This section deals with actions of modular groups on isotopy classes
of NET maps.  We formally present results only for the modular group.
At the end of the section we discuss the analogs for the pure modular
group and for the extended modular group, which allows for reversal of
orientation.  Let $f\co S^2\to S^2$ be a NET map.  This will be fixed
for all of this section.  As in Section~\ref{sec:background}, let
$\zp_1,\zp_2\co \bbR^2\to S^2$ be branched covering maps so that $f\circ
\zp_1=\zp_2$.  The set $P=P_2$ of branch values of
$\zp_2$ in $S^2$ is the postcritical set of $f$.  Let $\zL_i$
be the lattice of branch points of $\zp_i$ in $\mathbb{R}^2$
for $i\in \{1,2\}$.

It is always possible to take the lattice $\zL_2$ to be
$\mathbb{Z}^2$.  We do so for simplicity.  We define the special
affine group $\text{SAff}(2,\mathbb{Z})$ to be the group of all
orientation-preserving affine isomorphisms $\zJ\co \mathbb{R}^2\to
\mathbb{R}^2$ such that $\zJ(\mathbb{Z}^2)=\mathbb{Z}^2$.  So
$\text{SAff}(2,\mathbb{Z})$ is the set of all maps $x\mapsto Ax+b$,
where $A\in \text{SL}(2,\mathbb{Z})$, $x$ is a column vector in
$\mathbb{R}^2$ and $b$ is a column vector in $\mathbb{Z}^2$.  As
usual, let $\zG_2$ be the group of all Euclidean isometries of the
form $x\mapsto 2\zl\pm x$ for some $\zl\in \mathbb{Z}^2$.  Let $G$ be
the modular group of the pair $(S^2,P)$.  Proposition 2.7 of Farb and
Margalit's book \cite{fm} shows that $G$ is a semidirect product with
normal subgroup isomorphic to $\zL_2/2\zL_2\cong
(\mathbb{Z}/2\mathbb{Z})^2$ and quotient group
$\text{PSL}(2,\mathbb{Z})$.  This is essentially the content of the
following proposition.

\begin{prop}\label{prop:mod} The branched map $\zp_2$ induces a
surjective group homomorphism $\text{SAff}(2,\mathbb{Z})\to G$ with
kernel $\zG_2$.  Hence $G\cong \text{SAff}(2,\mathbb{Z})/\zG_2.$
\end{prop}
  \begin{proof} We define a group homomorphism $\zv\co
\text{SAff}(2,\mathbb{Z})\to G$ as follows.  Let $\zJ\in
\text{SAff}(2,\mathbb{Z})$.  Lemma~2.1 with $\zL=\zL'=\mathbb{Z}^2$
obtains an orientation-preserving homeomorphism $\zj\co (S^2,P)\to
(S^2,P)$ such that $\zp_2\circ \zJ=\zj\circ \zp_2$.  We
define $\zv(\zJ)$ to be the isotopy class of $\zj$ in $G$.  Using
Lemma~2.2, we see that $\zv$ is surjective, that it is a group
homomorphism and that its kernel is $\zG_2$.  This proves
Proposition~\ref{prop:mod}.
\end{proof}

Because $\text{SAff}(2,\mathbb{Z})$ consists of maps of the form
$x\mapsto Ax+b$ with $A\in \text{SL}(2,\mathbb{Z})$ and $b\in
\mathbb{Z}^2$, Proposition~\ref{prop:mod} implies that every element
of $G$ has a $\text{PSL}(2,\mathbb{Z})$-term and a translation term.
The translation term is an element of $\mathbb{Z}^2$ modulo
$2\mathbb{Z}^2$.  We say that an element of $G$ is elliptic, parabolic
or hyperbolic according to whether its $\text{PSL}(2,\mathbb{Z})$-term
is elliptic, parabolic or hyperbolic.  (We define reflections and
glide reflections similarly when dealing with the extended modular
group.)  We say that an element of $G$ is a translation if its
$\text{PSL}(2,\mathbb{Z})$-term is trivial.  In this way every element
of the extended modular group has a type, that is, it is either
elliptic, parabolic, hyperbolic, a reflection, a glide reflection or a
translation.

Let $G_f$ denote the subgroup of $G$ consisting of those isotopy
classes represented by homeomorphisms $\zj\co (S^2,P)\to (S^2,P)$ for
which there exists a homeomorphism $\widetilde{\zj}\co (S^2,P)\to
(S^2,P)$ such that $\zj\circ f=f\circ \widetilde{\zj}$.  Hence the
bottom half of the diagram in Figure~\ref{fig:lift} is commutative.
We call $G_f$ the subgroup of \emph{liftables} in $G$.  Let
$\zD=\zp_1^{-1}(P)$.  We define the \emph{special affine
group} $\text{SAff}(f)$ of $f$ to be the group of all
orientation-preserving affine isomorphisms $\zJ\co \mathbb{R}^2\to
\mathbb{R}^2$ such that $\zJ(\zL_2)=\zL_2$, $\zJ(\zL_1)=\zL_1$ and
$\zJ(\zD)=\zD$.  The group $\text{SAff}(f)$ actually depends on
$\zL_2$, $\zL_1$ and $\zD$ in addition to $f$, but it is well defined
up to a conjugation isomorphism.  Its elements are affine planar
homeorphisms that first descend under the projection $\zp_1$
and then further descend under $f$, as in Figure \ref{fig:lift}.

\begin{prop}\label{prop:saff} The group homomorphism of
Proposition~\ref{prop:mod} restricts to a surjective group
homomorphism $\text{SAff}(f)\to G_f$.  In particular, every pair of
homeomorphisms $\zj$ and $\widetilde{\zj}$ stabilizing $P$ as in
Figure~\ref{fig:lift} can be modified by isotopies rel $P$ so that
they lift to an affine isomorphism $\zJ$ as in Figure~\ref{fig:lift}.
The kernel of this group homomorphism is $\text{SAff}(f)\cap \zG_2$.
So $G_f\cong \text{SAff}(f)/(\text{SAff}(f)\cap \zG_2)$.
\end{prop}
  \begin{proof} Let $\zJ\in \text{SAff}(f)$.  Lemma~2.1 obtains
homeomorphisms $\zj$ and $\widetilde{\zj}$ as in
Figure~\ref{fig:lift}.  Hence the group homomorphism of
Proposition~\ref{prop:mod} maps $\text{SAff}(f)$ to $G_f$.  Statement
2 of Lemma~2.2 proves the second statement of
Proposition~\ref{prop:saff}, giving surjectivity.  The kernel of this
restriction group homomorphism is clearly $\text{SAff}(f)\cap \zG_2$.
This proves Proposition~\ref{prop:saff}.

\begin{figure}
  \begin{equation*}
\xymatrix{\mathbb{R}^2\ar[d]_{\zp_1}\ar[r]^\zJ &
\mathbb{R}^2\ar[d]_{\zp_1}\\
S^2\ar[d]_f\ar[r]^{\widetilde{\zj}} & S^2\ar[d]_f\\
S^2\ar[r]^\zj & S^2}
  \end{equation*}
\caption{ Lifting $\zj$; descending $\Psi$}
\label{fig:lift}
\end{figure}
\end{proof}

The preceding discussion gives a simple necessary and sufficient
criterion for an element $\psi$ to lie in $G_f$. To prepare for the
statement, we introduce some terminology that will also feature in our
characterization of modular Hurwitz classes. Recall that we have
presented $f$ as in the beginning of \S \ref{sec:background}; we
therefore have lattices $\Lambda_2:=\mathbb{Z}^2$ and $\Lambda_1 <
\Lambda_2$.  There is a
distinguished subset $\mathcal{HS} \subset \mathcal{A}=\zL_2/2\zL_1$
 given by the image of
$\Delta:=\pi_1^{-1}(P_2)$ under the natural projection $\Lambda_2 \to
\Lambda_2/2\Lambda_1$. The subset $\mathcal{HS}$ is the \emph{Hurwitz
structure set} associated to this presentation of $f$.  The group
$\mathcal{A}$ and the subset $\mathcal{HS}$ may be visualized from a
presentation diagram as in Figure \ref{fig:firstprendgm} as
follows. Two copies of the parallelogram form a torus; the grid points
in this torus are the elements of $\mathcal{A}$, and the endpoints of
the green arcs in this torus are the elements of $\mathcal{HS}$.

\begin{lemma} \label{lemma:when_liftable} Suppose $\Psi \in
\text{SAff}(2,\mathbb{Z})$ and let $\psi \in G$ be the induced mapping
class element.
\begin{enumerate}
  \item $\psi \in G_f$ if and only if $\Psi(\Lambda_1)=\Lambda_1$ and
the induced map $\overline{\Psi}: \mathcal{A} \to \mathcal{A}$
satisfies $\overline{\Psi}(\mathcal{HS}) = \mathcal{HS}$.
  \item Suppose $\psi \in G_f \cap \text{PMod}(S^2, P_f)$. If there
exists a lift $\widetilde{\psi}$ which is the identity on $P_f$, it is
unique. This occurs if and only if $\Psi \in \Gamma(2)+2\Lambda_1$,
$\Psi(\Lambda_1)=\Lambda_1$, and the induced map $\overline{\Psi}:
\mathcal{A} \to \mathcal{A}$ has the property that
$\overline{\zJ}(h)=\pm h$ for every $h\in \cH \cS$.
\end{enumerate}
\end{lemma}
\begin{proof} The uniqueness in conclusion 2 follows from \cite[Lemma
2.1]{Ko} as well as the discussion for the pure modular group at
the end of this section. The rest is subsumed in the previous
discussion.
\end{proof}

The next proposition gives some rough bounds on the index of the
liftable subgroup.

\begin{prop}\label{prop:gamman} $(1)$ In terms of the isomorphism
$G\cong \text{PSL}(2,\bbZ)\ltimes (\bbZ/2\bbZ)^2$, the group $G_f$
contains $\overline{\zG}(2D)\times \{0\}$, where $D=\deg(f)$ and
$\overline{\zG}(2D)$ is the image in $\text{PSL}(2,\bbZ)$ of the
principle congruence subgroup $\zG(2D)=\{\left[\begin{smallmatrix}a &
b \\ c & d
\end{smallmatrix}\right]\in \text{SL}(2,\bbZ):
\left[\begin{smallmatrix}a &b \\ c & d \end{smallmatrix}\right]
\equiv \left[\begin{smallmatrix}1 & 0 \\ 0 & 1
\end{smallmatrix}\right]\text{ mod } 2D\}$.\\
$(2)$ 
  \begin{equation*}
[G:G_f]\le 16D^3\prod_{p|2D}^{}(1-p^{-2})
  \end{equation*}
\end{prop}
  \begin{proof} We begin with statement 1.  We have that
$D=\left|\zL_2/\zL_1\right|$.  So $D \zL_2\subseteq \zL_1$ and $2D
\zL_2\subseteq 2\zL_1$.  It follows that $\zG(2D)$ maps $2\zL_1$ into
itself and acts trivially on $\mathcal{A}$.  Hence $\zG(2D)\subseteq
\text{SAff}(f)$.  Proposition~\ref{prop:saff} now completes the proof
of statement 1.

Statement 2 results from multiplying 4
($=\left|(\bbZ/2\bbZ)^2\right|$) times the index of
$\overline{\zG}(2D)$ in $\text{PSL}(2,\bbZ)$, a well known number
\cite[Theorem 4.2.5]{M}.
\end{proof}

The group of \emph{modular group deck transformations} of $f$ is the
set $\text{DeckMod}(f)$ of elements in $G$ represented by
homeomorphisms $\zv\co (S^2,P)\to (S^2,P)$ such that
$f\circ\zv=\zj\circ f$, where $\zj\co (S^2,P)\to (S^2,P)$ is a
homeomorphism representing the trivial element of $G$.  In particular,
if $\zj$ is the identify map, then $\zv$ is an ordinary deck
transformation of $f$ (which stabilizes $P$).

\begin{prop}\label{prop:deck} $(1)$ The branched map $\zp_1$ induces a
surjective group homomorphism $\text{SAff}(f)\cap \zG_2\to
\text{DeckMod}(f)$ with kernel $\zG_1$.  Hence the elements of
$\text{DeckMod}(f)$ are translations and
  \begin{equation*}
\text{DeckMod}(f)\cong (\text{SAff}(f)\cap \zG_2)/\zG_1.
  \end{equation*}

\noindent $(2)$ There are three possibilities for $\text{DeckMod}(f)$
up to isomorphism:
  \begin{equation*}
\{1\}, \text{ }\mathbb{Z}/2\mathbb{Z} \text{ or }
(\mathbb{Z}/2\mathbb{Z})^2.
  \end{equation*}
\end{prop}
  \begin{proof} Let $\zJ\in \text{SAff}(f)\cap \zG_2$.  Lemma~2.1
shows that $\zJ$ induces homeomorphisms $\zj$ and $\widetilde{\zj}$ as
in Figure~\ref{fig:lift}, making the diagram commute.  Because $\zJ\in
\zG_2$, the map $\zj$ is the identity map.  So $\widetilde{\zj}$
represents an element of $\text{DeckMod}(f)$.  This defines a group
homomorphism from $\text{SAff}(f)\cap \zG_2$ to $\text{DeckMod}(f)$.
Statement 2 of Lemma~2.2 shows that this homomorphism is surjective.
Its kernel is clearly $\text{SAff}(f)\cap \zG_1=\zG_1$.  This proves
statement 1.

Statement 1 implies that $\text{DeckMod}(f)$ is isomorphic to a
subgroup of $\mathcal{A}$.
%$\zG_2/\zG_1\cong 2\mathcal{A}$.  
The latter group arises
as translations by elements of $2\zL_2$ modulo translations by
elements of $2\zL_1$.  Because these translations must map $\zL_1$
into itself, $\text{DeckMod}(f)$ is in fact isomorphic to a subgroup
of $(2 \zL_2\cap \zL_1)/2\zL_1\subseteq \zL_1/2\zL_1\cong
(\mathbb{Z}/2\mathbb{Z})^2$.  This easily proves statement 2 and
completes the proof of Proposition~\ref{prop:deck}.
\end{proof}

\begin{remark}\label{remark:deckmod} In this remark we discuss the fact
that each of the three groups in statement 2 of
Proposition~\ref{prop:deck} occurs for some NET map.  For this we use
the notion of Hurwitz structure set defined just before Lemma~\ref{lemma:when_liftable}.  The proof of statement 2 essentially
shows that $\text{DeckMod}(f)$ is isomorphic to the subgroup of
$(2\zL_2\cap \zL_1)/2\zL_1$ which stabilizes the Hurwitz structure set
$\cH\cS$ of $f$ under the translation action of $(2\zL_2\cap \zL_1)/2\zL_1$
on $\zL_2/2\zL_1$.  So if $\cH\cS$ is not a union of cosets of any of the
subgroups of order 2 in $(2\zL_2\cap \zL_1)/2\zL_1$, then
$\text{DeckMod}(f)$ is trivial.  This is the generic case.  In
particular, if $2\zL_2\cap \zL_1=2\zL_1$, then $\text{DeckMod}(f)$ is
trivial.  The condition $2\zL_2\cap \zL_1=2\zL_1$ is true exactly when
the degree $|\zL_2/\zL_1|$ of $f$ is odd.  So if the degree of $f$ is
odd, then $\text{DeckMod}(f)$ is trivial.

If $\cH\cS$ is a union of cosets of some subgroup of order 2 in
$(2\zL_2\cap \zL_1)/2\zL_1$ but not a larger subgroup of $(2\zL_2\cap
\zL_1)/2\zL_1$, then $\text{DeckMod}(f)\cong \mathbb{Z}/2\mathbb{Z}$.
This is the situation for the NET map $f$ of Example 10.7 of
\cite{cfpp}.  In that example we identify $\zL_2/2\zL_1$ with
$(\mathbb{Z}/4\mathbb{Z})\oplus (\mathbb{Z}/4\mathbb{Z})$ so that
$2\zL_2=\zL_1$, $\zL_1/2\zL_1=\{(0,0),(2,0),(0,2),(2,2)\}$ and
$\cH\cS=\{(0,0),\pm (1,0),(2,0),\pm (1,2)\}$.  The stabilizer of
$\cH\cS$ in $\zL_1/2\zL_1$ is $\{(0,0),(2,0)\}\cong
\mathbb{Z}/2\mathbb{Z}$.  A nontrivial modular group deck
transformation of $f$ is used in \cite{cfpp} to show that $\sigma_f$
is constant.

If $\zL_1\subseteq 2\zL_2$ and $\cH\cS$ is a union of cosets of
$\zL_1/2\zL_1$, then $\text{DeckMod}(f)\cong
(\mathbb{Z}/2\mathbb{Z})^2$.  This occurs for every Euclidean NET map
for which $\zL_1\subseteq 2\zL_2$ because the Euclidean case is
exactly the case in which $\cH\cS=\zL_1/2\zL_1$.
\end{remark}

We next discuss the modular group virtual multi-endomorphism in the
present setting.  Let $\zJ\in \text{SAff}(f)$.  We have two group
homomorphisms induced by the assignments $\zJ\mapsto \zj$ and
$\zJ\mapsto \widetilde{\zj}$ with $\zJ$, $\zj$ and $\widetilde{\zj}$
related as in Figure~\ref{fig:lift}.  Since the diagram in
Figure~\ref{fig:lift} is commutative, $\zj\circ f=f\circ
\widetilde{\zj}$.  If $\text{DeckMod}(f)$ is trivial, then the
assignment $\zj\mapsto \widetilde{\zj}$ determines a group
homomorphism from $G_f$ to $G$.  This is the modular group virtual
multi-endomorphism.  In general, $\text{DeckMod}(f)$ is nontrivial and
instead of the assignment $\zj\mapsto \widetilde{\zj}$ determining a
function, it determines a many-valued multifunction.  It maps every
element of $G_f$ to a right coset of $\text{DeckMod}(f)$.  In
particular, the identity element of $G_f$ maps to
$\text{DeckMod}(f)$. 

The next proposition shows that these right cosets are also left
cosets.  It also provides information on the image of translations
under the modular group virtual multi-endomorphism.

\begin{prop}\label{prop:deckmod} $(1)$ The image of $G_f$ under the
modular group virtual multi-endomorphism is a subgroup of $G$, and
$DeckMod(f)$ is a normal subgroup of it.\\ 
$(2)$ If the Thurston pullback map $\zs_f$ of $f$ is not constant,
then the modular group virtual multi-endomorphism of $f$ maps
translations to translations.
\end{prop}
  \begin{proof} It is easy to see that the image of $G_f$ under the
modular group virtual multi-endomorphism is a subgroup of $G$.  To
prove the normality statement, let $\zd\in \text{DeckMod}(f)$ and let
$\zg\in G_f$.  Let $[f]$ denote the isotopy class of $f$, the set of
NET maps which are isotopic to $f$ rel the postcritical set of $f$.
Then
  \begin{equation*}
[f]=[f]\widetilde{\zg}\widetilde{\zg}^{-1}=\zg[f]\widetilde{\zg}^{-1}=
\zg[f]\zd \widetilde{\zg}^{-1}=[f]\widetilde{\zg}\zd
\widetilde{\zg}^{-1}.
  \end{equation*}
Thus $\widetilde{\zg}\zd \widetilde{\zg}^{-1}\in \text{DeckMod}(f)$.
This proves statement 1.

To prove statement 2, let $\zv$ be a liftable translation in the
modular group of $f$.  Let $\zs_\zv$ and $\zs_{\widetilde{\zv}}$ be
the pullback maps on $\mathbb{H}$ induced by $\zv$ and
$\widetilde{\zv}$.  Then $\zs_f\circ
\zs_\zv=\zs_{\widetilde{\zv}}\circ \zs_f$.  But $\zs_\zv=1$ because
$\zv$ is a translation.  So $\zs_f=\zs_{\widetilde{\zv}}\circ \zs_f$.
This implies that $\zs_f$ maps $\mathbb{H}$ into the fixed point set
of $\zs_{\widetilde{\zv}}$.  So if $\zs_f$ is not constant, then
$\zs_{\widetilde{\zv}}=1$.  Hence $\widetilde{\zv}$ is a translation.
This proves Proposition~\ref{prop:deckmod}.
\end{proof}

\begin{remark}\label{remark:tlns} This remark deals with statement 2
of Proposition~\ref{prop:deckmod}.  Again let $f$ be the NET map of
Example 10.7 of \cite{cfpp}; its pullback map $\zs_f$ is constant.
The group $G_f$ contains a translation whose image under the modular
group virtual multi-endomorphism is elliptic.  So translations do not
always map to translations.
\end{remark}

We continue by considering other choices for the group $G$.  

Suppose
that $G$ is the pure modular group of the pair $(S^2,P)$.  So now the
elements of $G$ are isotopy classes of orientation-preserving
homeomorphisms $\zj\co S^2\to S^2$ which fix $P$ pointwise.  We
consider the analog of Proposition~\ref{prop:mod} in this situation.
Since $P$ pulls back to $\zL_2$ via $\zp_2$, the group $G$
pulls back to the elements of $\text{SAff}(2,\mathbb{Z})$ which fix
the elements of $\zL_2$ modulo $\zG_2$.  It follows that
$G\cong\overline{\zG}(2)=\zG(2)/\{\pm 1\}$, where $\zG(2)$ is the
subgroup of $\text{SL}(2,\mathbb{Z})$ consisting of those elements
which are congruent to 1 modulo 2.  We next consider
Proposition~\ref{prop:saff} in this situation.  In order for the map
$\zj$ in Figure~\ref{fig:lift} to be in the pure modular group, we
need the map $\zJ$ to be in the subgroup generated by $\zG(2)$ and
$2\zL_2$.  In order for $\widetilde{\zj}$ to be in the pure modular
group, we need $\zJ$ to be in $\text{SAff}_P(f)$, which consists of
those elements of $\text{SAff}(f)$ which not only map $\zD$ to itself,
but fix the elements of $\zD$ modulo $2\zL_1$.  So $G_f$ is isomorphic
to $\text{SAff}_P(f)\cap\overline{\zG}(2)$.  Using
Proposition~\ref{prop:deck}, we see that the subgroup of $G$ analogous
to $\text{DeckMod}(f)$ is trivial.  Because of this, we have a pure
modular group virtual endomorphism, a single valued group
homomorphism.

In this paragraph we briefly discuss the situation in which we allow
for reversal of orientation.  So now the elements of $G$ are isotopy
classes of homeomorphisms $\zj\co (S^2,P)\to (S^2,P)$ which need not
preserve orientation.  The analog of Proposition~\ref{prop:mod} is
that $G\cong \text{Aff}(2,\mathbb{Z})/\zG_2$, where
$\text{Aff}(2,\mathbb{Z})$ is defined in the straightforward way.
Similarly, Proposition~\ref{prop:saff} becomes $G_f\cong
\text{Aff}(f)/(\text{Aff}(f)\cap \zG_2)$.  Finally, we note that the
virtual multi-endomorphism is also meaningful in the present situation
and that Proposition~\ref{prop:deckmod} holds as stated in the present
situation.

\section{The extended modular group action on Teichm\"{u}ller space}
\label{sec:teich}\nosubsections

In this section we fix a NET map $f$ with postcritical set $P$.  Let
$G$ denote the extended modular group of the pair $(S^2,P)$.  As at
the beginning of Section~\ref{sec:mdrgp}, we represent elements $\zj$
of $G$ by affine isomorphisms $\zJ\co \mathbb{R}^2\to \mathbb{R}^2$
such that $\zJ(\mathbb{Z}^2)=\mathbb{Z}^2$.  The elements of $G$ act
on the Teichm\"{u}ller space of $(S^2,P)$ by pulling back complex structures on
$S^2\setminus P$.  We let $\zs_\zj$ denote the map on $\mathbb{H}$
induced by the Teichm\"{u}ller action of $\zj\in G$.  The next
proposition relates $\zJ$ and $\zs_\zj$.  It presents the content of
the second paragraph after Theorem 9.1 of \cite{cfpp}.

\begin{prop}\label{prop:teichmx} Let $\zj$ be an element of the
extended modular group which is represented by the affine map
$\zJ(x)=\left[\begin{smallmatrix}a & b \\ c & d
\end{smallmatrix}\right]x+t$ with $\left[\begin{smallmatrix}a & b \\
c & d \end{smallmatrix}\right]\in \text{GL}(2,\mathbb{Z})$ and $t\in
\mathbb{Z}^2$.  Then the Teichm\"{u}ller action of $\zj$ on
$\mathbb{H}$ is given by $\zs_\zj(z)=\frac{dz+b}{cz+a}$ if $\zj$
preserves orientation and $\zs_\zj(z)=\frac{d \overline{z}+b}{c
\overline{z}+a}$ if $\zj$ reverses orientation.
\end{prop}

\begin{cor}\label{cor:type} Every element $\zj$ of the extended modular
group which is not a translation has the same type as $\zs_\zj$
(elliptic, parabolic,$\ldots$).
\end{cor}
  \begin{proof} Since the type of $\zj$ is determined by trace and
determinant, this follows from Proposition~\ref{prop:teichmx}.
\end{proof}

Let $\widetilde{\zj}$ denote the image of $\zj$ under the extended
modular group virtual multi-endomorphism of $f$ for every $\zj\in
G_f$.  In general, the assignment $\zj\mapsto \widetilde{\zj}$ is a
multifunction rather than a single-valued function.  However, we have
the following proposition.

\begin{prop}\label{prop:fn} $(1)$ The assignment $\zj\mapsto
\zs_{\widetilde{\zj}}$ for $\zj\in G_f$ is a function, even if the
extended modular group virtual multi-endomorphism is not.\\ 
$(2)$ The assignment $\zs_\zj\mapsto \zs_{\widetilde{\zj}}$ for
$\zj\in G_f$ is a group homomorphism.
\end{prop}
  \begin{proof} We begin with statement 1.  The discussion of the
modular group virtual multi-endomorphism following
Proposition~\ref{prop:deck} shows that in general the extended modular
group virtual multi-endomorphism is a multifunction which assigns the
elements of a coset of $\text{DeckMod}(f)$ to an element of the
extended modular group.  Statement 1 of Proposition~\ref{prop:deck}
shows that the elements of $\text{DeckMod}(f)$ are translations.  So
they act trivially on $\mathbb{H}$.  Thus the assignment $\zj\mapsto
\zs_{\widetilde{\zj}}$ is a function.  This proves statement 1.

Now we consider statement 2.  Let $T$ be the normal subgroup of $G$
which consists of all translations.  Proposition~\ref{prop:teichmx}
shows that the map $\zj\mapsto \zs_\zj$ induces an injective group
antihomomorphism\footnote{$\sigma_{\psi_2 \circ \psi_1}=\sigma_{\psi_1} \circ \sigma_{\psi_2}$} from $G/T$ to $\text{PGL}(2,\mathbb{Z})$.  As we saw
above, the map from $G$ to $G/T$ eliminates the indeterminacy of the
modular group virtual multi-endomorphism.  As a result, the map
$\zs_\zj\mapsto \zs_{\widetilde{\zj}}$ for $\zj\in G_f$ is a group
homomorphism.

This proves Proposition~\ref{prop:fn}.
\end{proof}

The next theorem provides information about $\zs_{\widetilde{\zj}}$ in
terms of $\zs_\zj$.  We prepare for it with a definition.  Under the
standard action on the upper half-plane, the subgroup of
$\text{PSL}(2,\bbZ)$ which fixes an extended rational number $r$ is a
parabolic subgroup.  This subgroup is infinite cyclic with a positive
generator, that is, a generator which fixes $r$ and moves every other
element of $\bbR\cup \{\infty \}$ in the counterclockwise direction
relative to $\bbH$.

\begin{thm}\label{thm:viredmpp} Let $\zs_f\co \mathbb{H}\to
\mathbb{H}$ be the Thurston pullback map of $f$.  Suppose that $\zj\in
G_f$.  Then the following statements hold.
\begin{enumerate}
  \item If $\zs_\zj$ is elliptic, then $\zs_{\widetilde{\zj}}$ is
either trivial or elliptic.
  \item Suppose that $\zs_\zj$ is parabolic.  Let $s$ be the slope
such that $\zs_\zj$ fixes $-1/s$.  Let $m$ be the multiplier of $s$.
\begin{enumerate}
  \item If $m=0$, then $\zs_{\widetilde{\zj}}$ is either trivial or
elliptic.
  \item Suppose that $m\ne 0$, that $\zs_\zj$ is the $n$th power of
the positive generator of the parabolic subgroup of
$\text{PSL}(2,\bbZ)$ which fixes $-1/s$ and that the slope function of
$f$ maps slope $s$ to slope $t$.  Then $\zs_{\widetilde{\zj}}$ is the
$(mn)$-th power of the positive generator of the parabolic subgroup of
$\text{PSL}(2,\bbZ)$ which fixes $-1/t$.
\end{enumerate}
  \item Suppose that $\zs_\zj$ is hyperbolic.  If the orbifold of $f$
is hyperbolic, then the translation length of $\zs_{\widetilde{\zj}}$
is less than the translation length of $\zs_\zj$ and the absolute
value of the trace of $\zs_{\widetilde{\zj}}$ is less than the
absolute value of the trace of $\zs_\zj$.
  \item If $\zs_\zj$ is a reflection, then $\zs_{\widetilde{\zj}}$ is a
reflection.
  \item If $\zs_\zj$ is a glide reflection, then
$\zs_{\widetilde{\zj}}$ is either a reflection or a glide reflection.
If $\zs_{\widetilde{\zj}}$ is a glide reflection and the orbifold of
$f$ is hyperbolic, then the translation length of
$\zs_{\widetilde{\zj}}$ is less than the translation length of
$\zs_\zj$, and the absolute value of the trace of
$\zs_{\widetilde{\zj}}$ is less than the absolute value of the trace
of $\zs_\zj$.
\end{enumerate}
\end{thm}
  \begin{proof} Statement 2 of Proposition~\ref{prop:fn} shows that
assignment $\zs_\zj\mapsto \zs_{\widetilde{\zj}}$ is a group
homomorphism.  Hence statement 1 is clear.

To prove statement 2, suppose that $\zs_\zj$ is parabolic.  Let $s$
and $m$ be as in statement 2.  Since the pure modular group has finite
index in the modular group, there exists a positive integer $k$ such
that $\zj^k$ is in the pure modular group.  It therefore is a Dehn
twist about a simple closed curve with slope $s$.  Let
$m=\frac{c}{d}$, as in Section~\ref{sec:background}.  The proof of
Theorem 7.1 of \cite{cfpp} shows that the modular group virtual
endomorphism maps $\zs_{\zj^{dk}}$ to the positive generator of the
parabolic subgroup of $\text{PSL}(2,\bbZ)$ which fixes $-1/t$ raised
to the power $ckn$.  Hence if $c=0$, then $\zs_{\widetilde{\zj}}$ is
either trivial or elliptic.  If $c\ne 0$, then we conclude the proof
of statement 2 by taking $(dk)$-th roots.

To prove statement 3, suppose that both $\zs_\zj$ and the orbifold of
$f$ are hyperbolic.  We prove statement 3 by adapting the proof of
Proposition 6.5 of \cite{kps} to the present situation.  Because the
orbifold of $f$ is hyperbolic, Theorem 6.3 of \cite{kps} implies that
the modular group virtual multi-endomorphism of $f$ is not injective.
In fact, its kernel contains a hyperbolic element.  Hence $\zs_f$ is
not injective.  Not being an isometry, it strictly decreases
hyperbolic distances.  Let $d\co \mathbb{H}\times \mathbb{H}\to
\mathbb{R}$ denote a hyperbolic metric on $\mathbb{H}$.  Among all
elements $z\in \mathbb{H}$, the minimum value of $d(z,\zs_\zj(z))$ is
achieved exactly for the elements on the translation axis of
$\zs_\zj$.  It is the translation length of $\zs_\zj$.  Let $z$ be an
element of the translation axis of $\zs_\zj$, and let $w=\zs_f(z)$.
Then
  \begin{equation*}
d(w,\zs_{\widetilde{\zj}}(w))=
d(\zs_f(z),\zs_{\widetilde{\zj}}(\zs_f(z)))=
d(\zs_f(z),\zs_f(\zs_\zj(z)))<d(z,\zs_\zj(z)).
  \end{equation*}
Thus the translation length of $\zs_{\widetilde{\zj}}$ is less than
the translation length of $\zs_\zj$.

We next prove that the absolute value of the trace of
$\zs_{\widetilde{\zj}}$ is less than the absolute value of the trace
of $\zs_\zj$.  This is clear if $\zs_{\widetilde{\zj}}$ is not
hyperbolic, so we assume that $\zs_{\widetilde{\zj}}$ is hyperbolic.
We use the fact that, up to a global multiplicative constant, the
translation length of a hyperbolic element $\zg$ of
$\text{SL}(2,\mathbb{Z})$ is $\ln(\left|\zl\right|)$, where $\zl$ is
the eigenvalue of $\zg$ with larger absolute value.  The absolute
value of the trace of $\zg$ is $\left|\zl+\zl^{-1}\right|\in
\mathbb{Z}$.  So decreasing $\left|\zl\right|$ decreases the absolute
value of the trace of $\zg$.  It follows that the absolute value of
the trace of $\zs_{\widetilde{\zj}}$ is less than the absolute value
of the trace of $\zs_\zj$.  This proves statement 3.

To prove statement 4, note that if $\zs_\zj$ is a reflection, then the
equation $\zs_f\circ \zs_\zj=\zs_{ \widetilde{\zj}}\circ \zs_f$
implies that $\zs_{\widetilde{\zj}}$ reverses orientation.  Statement
2 of Proposition~\ref{prop:fn} implies that the assignment
$\zs_\zj\mapsto \zs_{\widetilde{\zj}}$ is a group homomorphism.  So
$\zs_{\widetilde{\zj}}$ is an involution which reverses orientation.  It
must be a reflection.

Statement 5 can be proved much as statements 3 and 4 were proved.

This completes the proof of Theorem~\ref{thm:viredmpp}.
\end{proof}

As a nice application of the material in this section and the previous
one, we prove the following theorem.  This theorem easily generalizes
to all Thurston maps with four postcritical points. It complements
\cite[Theorem 9.2]{kps}, which treats only twisting by full Dehn
twists.

\begin{thm}\label{thm:twistobstrn} Let $f$ be a non-Euclidean NET map with an
obstruction whose multiplier is 1.  Let $\zw$ be an element in the
modular group of $f$ which together with the translations generates
the stabilizer of the obstruction of $f$.  Then the homotopy classes
$[f]\zw^n$ are mutually Thurston inequivalent for $n\in \bbZ$.
\end{thm}
  \begin{proof} Let $m$ and $n$ be integers such that $[f]\zw^m$ and
$[f]\zw^n$ are Thurston equivalent.  Then there exists an element
$\zv$ in the modular group of $f$ such that $\zv[f]\zw^m
\zv^{-1}=[f]\zw^n$.  Now we use the fact that if $f_1$ and $f_2$ are
obstructed NET maps with $f_2=\zj f_1\zj^{-1}$, where $\zj$ is a
homeomorphism representing an element of the modular group of $f_1$,
then $\zj$ maps the obstruction of $f_1$ to the obstruction of $f_2$.
Hence $\zv$ fixes the obstruction of $f$.  This and the assumptions
imply that $\zv$ is a power of $\zw$ up to translations, and so $\zv$
and $\zw$ commute up to translations.  Next we note that the equation
$\zv[f]\zw^m \zv^{-1}=[f]\zw^n$ implies that $\zv$ is liftable
relative to $f$.  So $[f]\widetilde{\zv}\zw^m \zv^{-1}\zw^{-n}=[f]$,
where $\widetilde{\zv}$ is one modular group element in the image of
$\zv$ under the modular group virtual multi-endomorphism of $f$.
Because the obstruction of $f$ has multiplier 1, statement 2 of
Theorem~\ref{thm:viredmpp} implies that $\widetilde{\zv}=\zv \zt $ for
some translation $\zt$ .  Hence $[f]\zw^{m-n}\zt'=[f]$ for some
translation $\zt'$.  Hence $\zw^{m-n}\zt'\in \text{DeckMod}(f)$.  But
Proposition~\ref{prop:deck} implies that $\text{DeckMod}(f)$ consists
of translations.  Thus $m=n$, which proves
Theorem~\ref{thm:twistobstrn}.
\end{proof}

\section{Hurwitz classes and actions of modular groups }
\label{sec:hurwitza}\nosubsections

This section introduces Hurwitz classes of NET maps and considers
related actions of modular groups.  There are various definitions of
Hurwitz equivalence in the literature.  We give definitions which
are suitable for our study of NET maps.  

We first recall some definitions from the introduction, and give some related terminology. Let $f\co S^2\to S^2$ and $f'\co S^2\to S^2$ be NET maps with
postcritical sets $P$ and $P'$.  Let $\zv,\zj\co (S^2,P)\to (S^2,P')$
be orientation-preserving homeomorphisms such that $\zv\circ f=f'\circ
\zj$.  In this situation we say that $f$ and $f'$ are \emph{Hurwitz
equivalent for the modular group}.  This defines an equivalence
relation on the set of NET maps.  The equivalence classes of this
equivalence relation are called \emph{modular group Hurwitz
classes}.  We denote the modular group Hurwitz class of $f$ by
$\text{Hurw}(f)$.  If in addition $\zv$ and $\zj$ agree on $P$, then
we say that $f$ and $f'$ are \emph{Hurwitz equivalent for the pure
modular group}.  The equivalence classes of this equivalence relation
are called \emph{pure modular group Hurwitz classes}.  We denote the
pure modular group Hurwitz class of $f$ by $\text{PHurw}(f)$.

We let $\text{Hurw}_P(f)$ denote the subset of $\text{Hurw}(f)$
consisting of those maps $f'$ with $P'=P$.  The group
$\text{Homeo}(S^2,P)$ of all orientation-preserving homeomorphisms
from $S^2$ to $S^2$ which map $P$ to itself acts on $\text{Hurw}_P(f)$
both on the left and the right.  So does the subgroup
$\text{Homeo}_0(S^2,P)$ consisting of those elements which are
isotopic rel $P$ to the identity map.  We let
  \begin{equation*}
\text{HurwMod}_P(f)=\text{Homeo}_0(S^2,P)\backslash\text{Hurw}_P(f)/
\text{Homeo}_0(S^2,P).
  \end{equation*}
We let $[f]$, the isotopy class of $f$, denote the element of
$\text{HurwMod}_P(f)$ represented by $f$.  Similarly,
  \begin{equation*}
\text{PHurwMod}_P(f)=\text{Homeo}_0(S^2,P)\backslash\text{PHurw}_P(f)/
\text{Homeo}_0(S^2,P).
  \end{equation*}

The modular group $\text{Mod}(S^2,P)$ acts on $\text{HurwMod}_P(f)$
both on the left and the right by post- and pre-composition,
respectively.  Equipped with these commuting actions,
$\text{HurwMod}_P(f)$ is a $\text{Mod}(S^2,P)$-biset, in the sense of
\cite{N}\footnote{The term `permutational bimodule' is used there
instead of `biset'.}.  Since the right action need not be free, the
algebraic structure is more complicated.  Thus we have a fundamental
tension: pure modular group Hurwitz classes behave better
algebraically, while modular group Hurwitz classes are in some sense
more fundamental.

Just as for NET maps, we let $[\zv]$ denote the isotopy class rel $P$
of a homeomorphism $\zv\co (S^2,P)\to (S^2,P)$.  We define the left
and right stabilizers of $[f]$ in $\text{Mod}(S^2,P)$ in the
straightforward way.  Similarly, the pure modular group
$\text{PMod}(S^2,P)$ acts on $\text{PHurwMod}_P(f)$ both on the left
and the right.  We define left and right stabilizers of $[f]$ in
$\text{PMod}(S^2,P)$ as before.

There is a modular group virtual multi-endomorphism of
$\text{Mod}(S^2,P)$ and a pure modular group virtual endomorphism of
$\text{PMod}(S^2,P)$.  By the kernel of the modular group virtual
multi-endomorphism we mean the set of all isotopy classes in
$\text{Mod}(S^2,P)$ for which some image under the modular group
virtual multi-endomorphism is trivial.

The following theorem presents some results concerning the action of
the modular group on $\text{HurwMod}_P(f)$.  Essentially the same
results hold for pure modular Hurwitz classes.  We discuss this very
briefly after the proof.

\begin{thm}\label{thm:action} Let $f$ be a NET map with postcritical
set $P$.  Let $G=\text{Mod}(S^2,P)$, and let $G_f$ be the group of
liftables in $G$ as usual.
\begin{enumerate}
  \item The left stabilizer of $[f]$ in $G$ is the kernel of the
modular group virtual multi-endomorphism.
  \item The right stabilizer of $[f]$ in $G$ is $\text{DeckMod}(f)$.
  \item If $\zg\in G_f$, then $\zg[f]=[f]\widetilde{\zg}$, where
$\widetilde{\zg}$ is any image of $\zg$ under the modular group
virtual multi-endomorphism.
  \item Let $\zj_1,\dotsc,\zj_k$ form a left transversal for $G_f$ in
$G$, and let $\zg_1,\zg_2\in G$.  Then $\zj_i[f]\zg_1=\zj_j[f]\zg_2$ for
some $i,j\in \{1,\dotsc,k\}$ if and only if $i=j$ and
$\zg_2\zg_1^{-1}\in \text{DeckMod}(f)$.
  \item The right action of $G$ on $\text{HurwMod}_P(f)$ has finitely
many orbits.
  \item Let $\zv,\zj\co (S^2,P)\to (S^2,P)$ be orientation-preserving
homeomorphisms.  Then $f\circ \zv$ is Thurston equivalent to $f\circ
\zj$ if and only if $[\zv]=\widetilde{\zg}[\zj]\zg^{-1}$ for some
$\zg\in G_f$ and some image $\widetilde{\zg}$ of $\zg$ under the
modular group virtual multi-endomorphism.
\end{enumerate}
\end{thm}
  \begin{proof} Statements 1, 2 and 3 are clear.

To prove statement 4, suppose that $\zj_i[f]\zg_1=\zj_j[f]\zg_2$ for
some $i,j\in \{1,\dotsc,k\}$.  Then
$\zj_j^{-1}\zj_i[f]=[f]\zg_2\zg_1^{-1}$.  This means that
$\zj_j^{-1}\zj_i\in G_f$.  So $i=j$.  Hence $[f]=[f]\zg_2\zg_1^{-1}$.
This means that $\zg_2\zg_1^{-1}\in \text{DeckMod}(f)$.  This easily
proves statement 4.

To prove statement 5, we use the fact that every element of
$\text{HurwMod}_P(f)$ has the form $\zv[f]\zj$ for some $\zv,\zj\in
G$.  Hence statement 5 follows from statement 3 and statement 2 of
Proposition~\ref{prop:gamman}, which implies that $G_f$ has finite
index in $G$.

To prove statement 6, let $\zv,\zj\co (S^2,P)\to (S^2,P)$ be
orientation-preserving homeomorphisms such that $f\circ \zv$ is
Thurston equivalent to $f\circ \zj$.  It follows that $[f\circ
\zv]=[f][\zv]=\zg[f\circ \zj]\zg^{-1}=\zg[f][\zj]\zg^{-1}$ for some
$\zg\in G$.  This easily implies that $\zg\in G_f$.  Now statement 3
implies that $[f][\zv]=[f]\widetilde{\zg}[\zj]\zg^{-1}$.  Now Statement
4 implies that there exists $\ze\in \text{DeckMod}(f)$ such that
$\widetilde{\zg}[\zj]\zg^{-1}[\zv]^{-1}=\ze$.  So
$[\zv]=\ze^{-1}\widetilde{\zg}[\zj]\zg^{-1}$.  Since
$\ze^{-1}\widetilde{\zg}$ is an image of $\zg$ under the modular group
virtual multi-endomorphism, this proves the forward implication of
statement 6.  The backward implication is now clear.

This proves Theorem~\ref{thm:action}.
\end{proof}

Theorem~\ref{thm:action} is essentially valid when $\text{Mod}(S^2,P)$
is replaced by $\text{PMod}(S^2,P)$ and $\text{HurwMod}_P(f)$ is
replaced by $\text{PHurwMod}_P(f)$.  In statements 2 and 4 the group
$\text{DeckMod}(f)$ is replaced by the trivial group.  Statement 6
applied to $\text{Mod}(S^2,P)$ already handles the case in which $\zv$
and $\zj$ represent elements of $\text{PMod}(S^2,P)$: even if $\zv$
and $\zj$ represent elements of $\text{PMod}(S^2,P)$, all that can be
said about the element $\zg$ is that it represents an element of
$\text{Mod}(S^2,P)$.

The next theorem provides an effective procedure for calculating the
modular group virtual multi-endomorphism.  Effectiveness comes from the fact that the slope function $\mu_f$ is effectively computable \cite[\S 5]{cfpp}.

\begin{thm}\label{thm:viredmp} Let $f$ be a NET map, and suppose that
the usual lattice $\zL_2$ is $\mathbb{Z}^2$.  We calculate slopes of
simple closed curves in $S^2\setminus P_2$ using the standard basis
vectors $e_1$ and $e_2$ of $\mathbb{Z}^2$ as usual.  Let $\zm_f$ be
the slope function of $f$ and suppose that there exist slopes $s$ and
$t$ such that $\zm_f(s)$ and $\zm_f(t)$ are distinct slopes.  Let
$\zJ\in \text{Aff}(f)$, and let $\widetilde{\zj}\co S^2\to S^2$ and
$\zj\co S^2\to S^2$ be the homeomorphisms induced by $\zp_1$
and $\zp_2$ as in Figure~\ref{fig:lift}.  Let
$\widetilde{\zJ}\co \mathbb{R}^2\to \mathbb{R}^2$ be an affine
isomorphism such that the homeomorphism of $S^2$ which
$\widetilde{\zJ}$ induces on $S^2$ via $\zp_2$ is isotopic to
$\widetilde{\zj}$ rel $P_2$.
\begin{enumerate}
  \item The linear part $x\mapsto Qx$ of $\widetilde{\zJ}$ can be
computed as follows.  Suppose that $\zJ$ takes lines with slope $s$,
respectively $t$, to lines with slope $s'$, respectively $t'$.  Then
$Q$ is the element of $\text{GL}(2,\mathbb{Z})$, unique up to
multiplication by $\pm 1$ whose determinant equals the determinant of
$\zJ$ and which takes lines with slope $\zm_f(s)$, respectively
$\zm_f(t)$, to lines with slope $\zm_f(s')$, respectively $\zm_f(t')$.
  \item The translation part of $\widetilde{\zJ}$ can be computed as
follows.  Let $T=\{0,e_1,e_2,e_1+e_2\}$, and for $\zt\in T$ let
$x_\zt=\zp_2(\zt)$.  So the postcritical set of $f$ is
$\{x_0,x_{e_1},x_{e_2},x_{e_1+e_2}\}$.  Then the translation term of
$\widetilde{\zJ}$ is the element $\zt\in T$ for which $\zJ$ maps
$\zp_1^{-1}(x_0)$ to $\zp_1^{-1}(x_\zt)$.
\end{enumerate}
\end{thm}
  \begin{proof} We first prove statement 1.  Let $\zv_s\co
(S^2,P_2)\to (S^2,P_2)$ be a Dehn twist (possibly twisting many times)
about a simple closed curve in $S^2\setminus P_2$ with slope $s$.
Since $\zJ$ takes lines with slope $s$ to lines with slope $s'$, the
homeomorphism $\zv_{s'}=\zj\circ \zv_s\circ \zj^{-1}$ is a Dehn twist
about a simple closed curve in $S^2\setminus P_2$ with slope $s'$.  We
next apply the argument which immediately precedes Theorem 7.1 of
\cite{cfpp}.  This argument shows that the modular group virtual
multi-endomorphism maps the isotopy class of some power of $\zv_s$ to
the isotopy class of a Dehn twist $\widetilde{\zv}_s$ (possibly
twisting many times) about a simple closed curve in $S^2\setminus P_2$
with slope $\zm_f(s)$.  A corresponding statement holds for
$\zv_{s'}$.  We replace $\zv_s$ and $\zv_{s'}$ by appropriate powers
to obtain the diagram in Figure~\ref{fig:endo}.  Thus
$\widetilde{\zj}$ takes simple closed curves in $S^2\setminus P_2$
with slope $\zm_f(s)$ to simple closed curves in $S^2\setminus P_2$
with slope $\zm_f(s')$.  So $\widetilde{\zJ}$, equivalently $Q$, takes
lines with slope $\zm_f(s)$ to lines with slope $\zm_f(s')$.
Similarly, $Q$ takes lines with slope $\zm_f(t)$ to lines with slope
$\zm_f(t')$.  Because $f$ preserves orientation, it is clear that $Q$
must have determinant 1 if $\zJ$ preserves orientation and it must
have determinant $-1$ if $\zJ$ reverses orientation.

\begin{figure}
  \begin{equation*}
\xymatrix{\zv_s\ar@{|->}[d]\ar@{|->}[r] & \widetilde{\zv}_s\ar@{|->}[d]\\
\zv_{s'}=\zj\circ \zv_s\circ \zj^{-1}\ar@{|->}[r] &
\widetilde{\zv}_{s'}=\widetilde{\zj}\circ \widetilde{\zv}_s\circ
\widetilde{\zj}^{-1}} 
  \end{equation*}
 \caption{Proving Theorem~\ref{thm:viredmp}}
\label{fig:endo}
\end{figure}

To see that this determines $Q$ up to multiplication by $\pm 1$, let
$R$ be an element of $\text{GL}(2,\mathbb{Z})$ with
$\text{det}(R)=\text{det}(Q)$ which takes lines with slope $\zm_f(s)$,
respectively $\zm_f(t)$, to lines with slope $\zm_f(s')$, respectively
$\zm_f(t')$.  Then $R^{-1}Q$ takes lines with slope $\zm_f(s)$,
respectively $\zm_f(t)$, to lines with slope $\zm_f(s)$, respectively
$\zm_f(t)$. This means that the linear fractional transformation
associated to $R^{-1}Q$ arises from an element of $\text{PSL}(2,\bbZ)$
and it fixes two extended rational numbers.  Hence this linear
fractional transformation is the identity map.  This proves statement
1 of Theorem~\ref{thm:viredmp}.

Statement 2 is rather clear, and so the proof of
Theorem~\ref{thm:viredmp} is complete.
\end{proof}

\section{Hurwitz invariants and elementary divisors }
\label{sec:hurwitzb}\nosubsections

The main goal of this section is to find a rather simple complete
combinatorial invariant for NET map modular group Hurwitz
classes.  This is achieved in Theorem~\ref{thm:hurwitz}.  Then we show
that elementary divisors form a complete set of invariants for
topological equivalence; this is Theorem \ref{thm:eleydivrs}.

We begin with a discussion of elementary divisors.  Let $f$ be a NET
map.  From a presentation for $f$ we have lattices $\zL_1\subseteq
\zL_2$ as usual; we have the group $\mathcal{A}=\Lambda_2/2\Lambda_1$
and the Hurwitz structure set $\mathcal{HS}$ from \S \ref{sec:mdrgp}.
There exist unique positive integers $m$ and $n$ with $n|m$ such that
$\zL_2/\zL_1\cong (\mathbb{Z}/m \mathbb{Z})\oplus (\mathbb{Z}/n
\mathbb{Z})$.  The integers $m$ and $n$ are the \emph{elementary
divisors} of $f$, defined in Section 8 of \cite{fpp1}.  Here is
another way to view the elementary divisors of $f$.  We factor $f$ in
the usual way as $f=h\circ g$, where $g$ is a Euclidean NET map and
$h$ is a homeomorphism.  The map $g$ is induced by an
orientation-preserving affine isomorphism $\zF\co \mathbb{R}^2\to
\mathbb{R}^2$ such that $\zF(\zL_2)=\zL_1$.  Expressing $\zF$ in terms
of a basis for $\zL_2$, we have that $\zF(x)=Ax+b$, where $A$ is a
$2\times 2$ matrix of integers and $b\in \zL_1$.  There exist matrices
$Q,R\in \text{SL}(2,\mathbb{Z})$ such that $A=QDR$, where $D$ is the
diagonal matrix whose diagonal entries are $m$ and $n$.  So $m$ and
$n$ are the elementary divisors of $A$.  Since
$\deg(f)=\deg(g)=\text{det}(A)=mn$, the elementary divisors of $f$
give finer information than $\deg(f)$.  We will see in
Corollary~\ref{cor:eleydivrs} that $m$ and $n$ are uniquely determined
by $f$. The \emph{atypical cases} are $(m,n)\in\{(2,1), (2,2)\}$.

We continue with a discussion of Hurwitz structure sets.  The
Hurwitz structure set of $f$, defined in Section~\ref{sec:mdrgp},
depends on the presentation.  Here, we analyze the dependency.  We say
that a subset $\cH\cS$ of $\mathcal{A}$ is a \emph{Hurwitz structure
set} if it is a disjoint union of four subsets of the form
$\cH\cS=\{\pm h_1\}\amalg \{\pm h_2\}\amalg \{\pm h_3\}\amalg \{\pm
h_4\}$, where $h_i\in \mathcal{A}$.  It is possible that $h_i=-h_i$,
but $h_i\ne \pm h_j$ if $i\ne j$ for $i,j\in \{1,2,3,4\}$.  The
Hurwitz structure set coming from a NET map presentation gives such an
example.  In fact, every such set $\mathcal{HS}\subseteq \cA$ is
the Hurwitz structure set associated to some NET map.  Now suppose
that we also have possibly different lattices $\zL'_1\subseteq
\zL'_2$.  Let $\zJ\co \zL_2\to \zL'_2$ be an orientation-preserving
affine isomorphism such that $\zJ(\zL_1)=\zL'_1$.  Then $\zJ$ induces
an affine bijection $\overline{\zJ}\co \mathcal{A}\to \mathcal{A}'$.
The linear part of $\overline{\zJ}$ takes Hurwitz structure sets in
$\mathcal{A}$ to Hurwitz structure sets in $\mathcal{A}'$.  Since
$\zJ(0)\in \zL'_1$, the translation part of $\overline{\zJ}$ is an
element of order 2.  It easily follows that the translation part of
$\overline{\zJ}$ takes Hurwitz structure sets in $\mathcal{A}'$ to
Hurwitz structure sets in $\mathcal{A}'$.  So $\overline{\zJ}$ takes
Hurwitz structure sets to Hurwitz structure sets.  We say that Hurwitz
structure sets $\cH\cS\subseteq \mathcal{A}$ and $\cH\cS'\subseteq
\mathcal{A}'$ are \emph{equivalent} if there exists such an affine
isomorphism $\zJ\co \zL_2\to \zL'_2$ with $\zJ(\zL_1)=\zL'_1$ such
that $\cH\cS'=\overline{\zJ}(\cH\cS)$.  This defines an equivalence
relation on Hurwitz structure sets.

The \emph{Hurwitz invariant} of $f$ is defined as the equivalence class of the Hurwitz structure set $\mathcal{HS}$ under the equivalence relation of the preceding paragraph.  It is a finer invariant than the elementary divisors. 
%Since
%$\mathcal{A}$ determines the elementary divisors of $f$, we view the
%Hurwitz invariant of $f$ as determining the elementary divisors of
%$f$, and so it gives finer information.

\begin{thm}\label{thm:hurwitz} The Hurwitz invariant of a NET map $f$
is an invariant of the modular group Hurwitz class of $f$.
Conversely, two modular group Hurwitz classes of NET maps with equal
Hurwitz invariants are equal.  Thus the Hurwitz invariant is a
complete invariant of modular group Hurwitz classes of NET maps.
\end{thm}
  \begin{proof} We first prove that Hurwitz invariants are indeed
invariants of modular group Hurwitz classes.  Let $f$ and $f'$ be NET
maps which are Hurwitz equivalent for the modular group.  Let $p_1$,
$q_1$, $p_2$, $q_2$, $\ldots $ be the usual maps, lattices, groups and
four-element sets for $f$.  We likewise have maps, lattices, groups
and four-element sets for $f'$.  Because $f$ and $f'$ are Hurwitz
equivalent for the modular group, there exist orientation-preserving
homeomorphisms $\zj,\zq\co (S^2,P_2)\to (S^2,P'_2)$ such that
$\zq\circ f=f'\circ \zj$.  This implies that $\zj(P_1)=P'_1$.  Now
statement 1 of Lemma~2.2 implies that there exist
orientation-preserving homeomorphisms $\zJ, \zQ\co \mathbb{R}^2\to
\mathbb{R}^2$ such that $\zJ(\zL_1)=\zL'_1$, $\zQ(\zL_2)=\zL'_2$, the
restriction of $\zJ$ to $\zL_1$ is affine, the restriction of $\zQ$ to
$\zL_2$ is affine, $\zp'_1\circ \zJ=\zj\circ \zp_1$ and $\zp'_2\circ
\zQ=\zq\circ \zp_2$.  Hence every face of the cube in
Figure~\ref{fig:hurwitzup} is commutative except possibly the top.

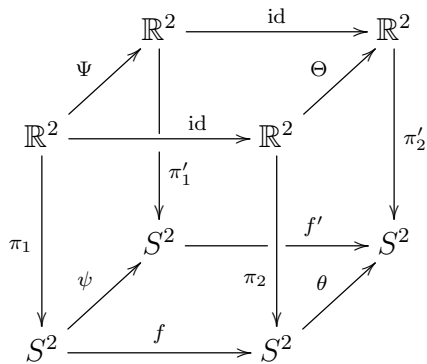
\begin{figure}
  \begin{equation*}
\xymatrix{ & \mathbb{R}^2\ar'[d][dd]^>>>>>>{\zp'_1}\ar[rr]^{\text{id}} & & \mathbb{R}^2\ar[dd]^{\zp'_2}\\
\mathbb{R}^2\ar[dd]_{\zp_1}
\ar[rr]^>>>>>>>>>>{\qquad\text{id}}\ar[ur]^\zJ &
 & \mathbb{R}^2\ar[dd]_>>>>>>{\zp_2}\ar[ur]^\zQ & \\
 & S^2\ar'[r][rr]^<<<<{f'} & & S^2\\
S^2\ar[rr]^f\ar[ur]^\zj & & S^2\ar[ur]^\zq &}
  \end{equation*}
 \caption{ Proving the forward direction of
Theorem~\ref{thm:hurwitz}}
\label{fig:hurwitzup}
\end{figure}

Now we consider the top of the cube in Figure~\ref{fig:hurwitzup}.
Since every face of the cube other than the top is commutative,
$\zp'_2\circ \zJ=\zp'_2\circ \zQ$.  It follows that $\zJ$ is $\zQ$
postcomposed with a deck transformation for $\zp'_2$, namely, an
element of $\zG'_2$.  So $\zJ(\zL_2)=\zL'_2$, and the restriction of
$\zJ$ to $\zL_2$ is affine.  So $\zJ$ induces an affine bijection
$\overline{\zJ}\co \mathcal{A}\to \mathcal{A}'$.  Since
$\zj(P_2)=P'_2$, the map $\overline{\zJ}$ takes the Hurwitz structure
set $p_1^{-1}(P_2)$ of $f$ to the Hurwitz structure set
$(p'_1)^{-1}(P'_2)$ of $f'$.  Thus the Hurwitz invariant of $f$ equals
the Hurwitz invariant of $f'$.  This proves that the Hurwitz invariant
is indeed an invariant of the modular group Hurwitz class of a NET
map.

It remains to prove that if two modular group Hurwitz classes of NET
maps have equal Hurwitz invariants, then they are equal.  So let $f$
and $f'$ be NET maps with postcritical sets $P_2$ and $P'_2$ and equal
Hurwitz invariants.  So there exists an orientation-preserving affine
isomorphism $\zJ\co \zL_2\to \zL'_2$, which we extend to all of
$\mathbb{R}^2$, such that $\zJ(\zL_1)=\zL'_1$ and $\zJ$ induces an
affine bijection $\overline{\zJ}\co \mathcal{A}\to \mathcal{A}'$
which takes the Hurwitz structure set $p_1^{-1}(P_2)$ of $f$ to the
Hurwitz structure set $(p'_1)^{-1}(P'_2)$ of $f'$.  So Lemma~2.1
yields orientation-preserving homeomorphisms $\zj\co (S^2,P_2)\to
(S^2,P'_2)$ and $\widetilde{\zj}\co (S^2,P_1)\to (S^2,P'_1)$ such that
$\zj\circ \zp_2=\zp'_2\circ \zJ$ and $\widetilde{\zj}\circ
\zp_1=\zp'_1\circ \zJ$.  Hence every face of the cube of
Figure~\ref{fig:hurwitzdown} other than possibly the bottom is
commutative.  Since the vertical maps are surjective, the bottom is
also commutative.  Because $\zJ$ induces a map from the Hurwitz
structure set of $f$ to the Hurwitz structure set of $f'$, it follows
that $\widetilde{\zj}(P_2)=P'_2=\zj(P_2)$.  Hence $f$ and $f'$ are
modular group Hurwitz equivalent.

\begin{figure}
  \begin{equation*}
\xymatrix{ & \mathbb{R}^2\ar'[d][dd]^>>>>>>{\zp'_1}\ar[rr]^{\text{id}}
& & \mathbb{R}^2\ar[dd]^{\zp'_2}\\ \mathbb{R}^2\ar[dd]_{\zp_1}
\ar[rr]^>>>>>>>>>>{\qquad\text{id}}\ar[ur]^\zJ & &
\mathbb{R}^2\ar[dd]_>>>>>>{\zp_2}\ar[ur]^\zJ & \\ &
S^2\ar'[r][rr]^<<<<{f'} & & S^2\\
S^2\ar[rr]^f\ar[ur]^{\widetilde{\zj}} & & S^2\ar[ur]^\zj &}
  \end{equation*}
 \caption{ Proving the backward direction of
Theorem~\ref{thm:hurwitz}}
\label{fig:hurwitzdown}
\end{figure}
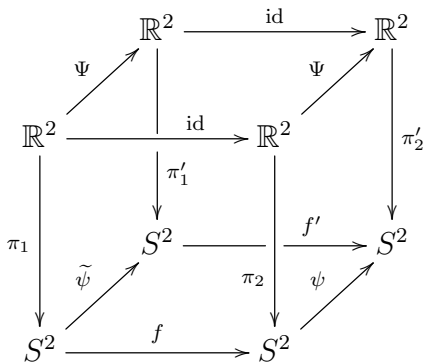

This proves Theorem~\ref{thm:hurwitz}.
\end{proof}

\begin{cor}\label{cor:eleydivrs} The elementary divisors of a NET map
$f$ are invariants of the modular group Hurwitz class of $f$.
\end{cor}

\begin{cor}\label{cor:dgl} Every modular group Hurwitz class of NET
maps is represented by a NET map whose presentation matrix is
diagonal. 
\end{cor}

\begin{cor}\label{cor:hbijection} In the typical cases, there is
a bijection between the set of Hurwitz classes of NET maps with
elementary divisors $(m,n)$ and the set of equivalence classes of
Hurwitz structure sets in groups of the form $(\mathbb{Z}/2m
\mathbb{Z})\oplus (\mathbb{Z}/2n \mathbb{Z})$, were $m$ and $n$ are
positive integers with $m\ge 2$ and $n|m$.
\end{cor}

\begin{remark}\label{remark:hurwbijn} In the atypical cases, there are
Hurwitz structure sets for which no Thurston map in the modular group
Hurwitz equivalence class which they determine has four postcritical
points.
\end{remark}

\begin{remark}\label{remark:enuhurw} Theorem~\ref{thm:hurwitz} can be
used to find one representative in every modular group Hurwitz class
of NET maps of small degree.  This is done in \cite{NET} through
degree 30.
\end{remark}

We next prove that the elementary divisors of a NET map form a
complete set of invariants for topological equivalence.  We note that
Thurston equivalence, pure modular group Hurwitz equivalence, modular
group Hurwitz equivalence and topological equivalence are successively
weaker equivalence relations.

\begin{thm}\label{thm:eleydivrs} Two NET maps are topologically
equivalent if and only if their elementary divisors are equal.
\end{thm}
  \begin{proof} Let $f$ and $f'$ be NET maps.  We first prove the
forward implication.

So suppose that $f$ and $f'$ are topologically equivalent.  Then there
exist homeomorphisms $\zv,\zj\co S^2\to S^2$ such that $\zv\circ
f=f'\circ \zj=\zj\circ (\zj^{-1}\circ f'\circ \zj)$.
Corollary~\ref{cor:eleydivrs} implies that $f'$ and $\zj^{-1}\circ
f'\circ \zj$ have the same elementary divisors, so we may assume that
$\zv\circ f=\zj\circ f'$.  Combining $\zv$ and $\zj$, we may even
assume that $f'=\zv\circ f$.  It follows that $\zv$ maps the critical
values of $f$ bijectively to the critical values of $f'$.  As
discussed in the introduction, a NET map has exactly two critical
values if and only if its elementary divisors are $(2,1)$ and it has
exactly three critical values if and only if its elementary divisors
are $(2,2)$.  So we may assume that we are not in one of these
atypical cases.  Thus we have that $\zv(P_f)=P_{f'}$.  Now let
$\zL_1$, $\zL_2$, $p_1$, $q_1$, $p_2$ and $q_2$ be usual lattices and
branched covering maps for $f$, so that $f\circ p_1\circ q_1=p_2\circ
q_2$.  Then
  \begin{equation*}
f'\circ p_1\circ q_1=\zv\circ f\circ p_1\circ q_1=(\zv\circ p_2)\circ
q_2.
  \end{equation*}
This and the fact that $\zv(P_f)=P_{f'}$ imply that $\zL_1$,
$\zL_2$, $p_1$, $q_1$, $\zv\circ p_2$ and $q_2$ are corresponding
lattices and branched covering maps for $f'$.  Since $\zL_1$ and
$\zL_2$ determine the elementary divisors for both $f$ and $f'$, the
elementary divisors of $f$ and $f'$ are equal.  This proves the
forward implication of Theorem~\ref{thm:eleydivrs}.

To prove the backward implication, suppose that the elementary
divisors of $f$ and $f'$ are equal.  Since every NET map is
topologically equivalent to a Euclidean NET map and topologically
equivalent NET maps have equal elementary divisors, we may assume that
$f$ and $f'$ are Euclidean.  Since the translation term in a NET map
presentation preserves equivalence, we may assume that $f$ and $f'$
are induced by maps which are not just affine, but even linear.  We
may also assume that $f$ and $f'$ have presentations with
$\zL_2=\zL'_2=\bbZ^2$.  Now we apply the well-known result that if $A$
and $A'$ are $2\times 2$ matrices of integers with positive
determinants and equal elementary divisors, then there exist $Q,R\in
\text{SL}(2,\bbZ)$ such that $A'=QAR$.  The matrices $Q$ and $R$
induce homeomorphisms of $S^2$ as in Lemma 2.1, which provide an
equivalence between $f$ and $f'$.

This proves Theorem~\ref{thm:eleydivrs}.
\end{proof}

The definition of Hurwitz invariant given above has the virtue of
expediency but not the virtue of naturality.  The point here is that
the definition of equivalence of Hurwitz structure sets in the finite
group $\mathcal{A}$ relies on presenting the group as $\zL_2/2\zL_1$
and lifting to $\zL_2$ rather than working in $\mathcal{A}$
intrinsically.  The rest of this section is devoted to giving another
definition of equivalence of Hurwitz structure sets which does not
rely on lifting to $\zL_2$.

For this we define the notion of \emph{special automorphism}.  Let
$\mathcal{A}$ be a finite Abelian group generated by two elements.
There exist unique positive integers $m$ and $n$ with $n|m$ such that
$\mathcal{A}\cong(\mathbb{Z}/m \mathbb{Z})\oplus (\mathbb{Z}/n
\mathbb{Z})$.  (The $m$ and $n$ here differ from the usual $m$ and $n$
by factors of 2.)  Let $\zv\co \mathcal{A}\to \mathcal{A}$ be a group
automorphism.  Then $\zv(n\mathcal{A})= n\mathcal{A}$, and so $\zv$
induces an automorphism of $\mathcal{A}/n\mathcal{A}\cong
(\mathbb{Z}/n\mathbb{Z})\oplus (\mathbb{Z}/n\mathbb{Z})$.
Automorphisms of the latter group are described by matrices just as
over a vector space.  In particular, determinants are meaningful.  By
choosing an isomorphism between $\mathcal{A}/n\mathcal{A}$ and
$(\mathbb{Z}/n \mathbb{Z})\oplus (\mathbb{Z}/n \mathbb{Z})$, the
automorphism $\zv$ induces an automorphism $\overline{\zv}$ of
$(\mathbb{Z}/n \mathbb{Z})\oplus (\mathbb{Z}/n \mathbb{Z})$.  The map
$\overline{\zv}$ depends on the choice of isomorphism between
$\mathcal{A}/n\mathcal{A}$ and $(\mathbb{Z}/n \mathbb{Z})\oplus
(\mathbb{Z}/n \mathbb{Z})$, but its determinant does not.  (This is
analogous to changing bases of a vector space.)  We say that $\zv$ is
a special automorphism if $\text{det}(\overline{\zv})=1$.  We
take this condition to be vacuously true when $n=1$.

Now that we have a definition, we can compute as follows.  Choose an
isomorphism between $\mathcal{A}$ and $(\mathbb{Z}/m \mathbb{Z})\oplus
(\mathbb{Z}/n \mathbb{Z})$.  Let $u$ and $v$ be the elements of
$\mathcal{A}$ such that $u$ corresponds to $(1,0)$ and $v$ corresponds
to $(0,1)$.  Let $\zv\co \mathcal{A}\to \mathcal{A}$ be a group
automorphism.  There exist integers $a$, $b$, $c$, $d$ so that
$\zv(xu+yv)=(ax+by)u+(cx+dy)v$ for all integers $x$ and $y$.  Then
$\zv$ is a special automorphism if and only if $ad-bc\equiv 1 \text{
mod } n$.

Now suppose that our finite Abelian group has the form
$\mathcal{A}=\zL_2/2\zL_1$, as usual, where $\zL_2=\mathbb{Z}^2$.
Suppose that $\zF\co \zL_2\to \zL_2$ is an orientation-preserving
affine isomorphism such that $\zF(\zL_1)=\zL_1$.  Then $\zF(x)=Ax+b$,
where $A\in \text{SL}(2,\mathbb{Z})$ and $b\in \zL_1$.  So the map
$x\mapsto Ax$ induces a special automorphism of $\mathcal{A}$, and the
map $x\mapsto x+b$ induces a translation by an element of order at
most 2.  Now we see how to define equivalence of Hurwitz structure
sets in $\mathcal{A}$ without lifting to $\zL_2$: Hurwitz structure
sets $\cH\cS$ and $\cH\cS'$ in $\mathcal{A}$ are equivalent if and
only if there exists a special automorphism $\zv\co \mathcal{A}\to
\mathcal{A}$ and an element $b\in \mathcal{A}$ with $2b=0$ such that
$\cH\cS'=\zv(\cH\cS)+b$.  This definition coincides with the
definition given above if every special automorphism of $\mathcal{A}$
lifts to $\zL_2$ because it is easy to lift the translation terms.
This is the content of the next lemma.   We have proved:

\begin{thm} \label{thm:hurwitz_special} The assignment $f \mapsto
\mathcal{HS}\subset \mathcal{A}$ induces a bijection between modular
Hurwitz classes of typical NET maps having elementary divisors
$(m,n)$ and orbits of Hurwitz structure sets under the action of the
group generated by the special automorphisms of $\mathcal{A}\cong 
(\bbZ/2m\bbZ)\oplus (\bbZ/2n \bbZ)$ and translations by elements of
order 2.
\end{thm}

\begin{lemma}\label{lemma:special}  Let $m$ and $n$ be
positive integers with $n|m$.  Let $\zv$ be a special automorphism of
$(\mathbb{Z}/m \mathbb{Z})\oplus (\mathbb{Z}/n \mathbb{Z})$.  Then
there exists an orientation-preserving group automorphism $\zF\co
\mathbb{Z}^2\to \mathbb{Z}^2$ which stabilizes the subgroup
$\left<(m,0),(0,n)\right>$ and induces the map $\zv$.  To state this
more concretely, let $a$, $b$, $c$, $d$ be integers such that
$\zv(1,0)=a(1,0)+c(0,1)$ and $\zv(0,1)=b(1,0)+d(0,1)$.  Then there
exists $\left[\begin{smallmatrix}a' & b' \\ c' & d'
\end{smallmatrix}\right]\in \text{SL}(2,\mathbb{Z})$ such that
$(a',b')\equiv (a,b) \mod m$ and $(c',d')\equiv (c,d)\mod n$.
\end{lemma}
  \begin{proof} We prove the concrete formulation of the lemma.  In
the situation of the concrete formulation, if $a=0$, then we may set
$a=m$.  Hence we may assume that $a\ne 0$.  The assumptions imply
that $ad-bc\equiv 1 \mod n$.

Now we define an integer $b'$ such that $b'\equiv b\mod m$, and in the
next paragraph we prove that $\gcd(a,b')=1$.  Because $a\ne 0$, only
finitely many primes divide $a$.  Using the Chinese remainder theorem,
we find an integer $z$ such that for every prime $p$ dividing $a$ we
have that
  \begin{equation*}
z\equiv \begin{cases}
0 \mod p &\text{ if }p\nmid b\\
1 \mod p &\text{ if }p\mid b.
\end{cases}
  \end{equation*}
We set $b'=b+zm$.

Now we show that $\gcd(a,b')=1$.  If $p$ is a prime with $p\mid a$ and
$p\nmid b$, then $p\mid zm$ and $p\nmid b'$.  It remains to consider
primes $p$ such that $p\mid a$ and $p\mid b$.  In this case $p\nmid n$
because $ad-bc\equiv 1\mod n$.  On the other hand, if $p\mid
\frac{m}{n}$, then $n\mid \frac{m}{p}$ and
  \begin{equation*}
\frac{m}{p}\zv(1,0)=\frac{m}{p}a(1,0)+\frac{m}{p}c(0,1)=0+0=0.
  \end{equation*}
This is impossible because $\zv(1,0)$ has order $m$.  Hence $p\nmid n$
and $p\nmid \frac{m}{n}$, and so $p\nmid m$.  Hence $p\nmid zm$, and
so $p\nmid b'$.  Therefore $\gcd(a,b')=1$.

Since $\gcd(a,b')=1$, there exist integers $x$ and $y$ such that
$ax+b'y=1$.  Set
  \begin{equation*}
c'=c-y(1-(ad-b'c)), \quad d'=d+x(1-(ad-b'c)) \quad\text{and}\quad a'=a.
  \end{equation*}
Then $(a',b')\equiv (a,b) \mod m$, $(c',d')\equiv
(c,d)\mod n$ and
  \begin{equation*}
\begin{aligned}
\left|\begin{matrix}a' & b' \\ c' & d' \end{matrix}\right| & =
a(d+x(1-(ad-b'c)))-b'(c-y(1-(ad-b'c)))\\
 & =ad-b'c+(ax+b'y)(1-(ad-b'c))=1.
\end{aligned}
  \end{equation*}

Thus we have a matrix in $\text{SL}(2,\mathbb{Z})$ which satisfies the
desired congruences.  We take $\zF$ to be the automorphism of
$\mathbb{Z}^2$ whose matrix with respect to the standard basis is this
matrix.  It remains to show that $\zF$ stabilizes
$\zL=\left<(m,0),(0,n)\right>$.  Because every element of $\zL$ is
divisible by $n$, it is clear that the second coordinate of every
element in $\zF(\zL)$ is divisible by $n$.  Because $\zv(0,1)$
has order $n$, the integer $b'$ must be divisible by $m/n$.  Now we see
that the first coordinate of every element in $\zF(\zL)$ is
divisible by $m$.

This proves Lemma~\ref{lemma:special}.
\end{proof}

\section{Computing the virtual multi-endomorphism}
\label{sec:exvme}

In this section, we illustrate the calculation of the virtual
multi-endomorphism in a concrete example.

We work with a NET map $f$ given by
circling one of the large dots in Figure~\ref{fig:PrenDgm}. Our
computations are independent of which vertex is circled.

The diagram determines (i) an ordered basis of $\zL_1$ given by
$\zl_1=(6,0)$ and $\zl_2=(0,1)$, and (ii) a fundamental domain for
$\Gamma_1$, which is the parallelogram drawn.  The vertices in this
fundamental domain which map to elements of $P_2=P_f$ under the
natural projection to $S^2_1$ are $(1,0)$, $(11,0)$, $(2,0)$,
$(10,0)$, $(1,1)$, $(11,1)$, $(2,1)$, and $(10,1)$.

We identify the pure modular group with $G=\overline{\zG}(2)$.  We
next find additional congruence conditions defining the subgroup $G_f$
of pure modular group liftables as follows.

Suppose $M=\left[\begin{smallmatrix}a & b \\ c &
d\end{smallmatrix}\right]\in\zG(2)$. 

For $M$ to stabilize $\zL_1$, we must have that $M \zl_1\in \zL_1$.
This imposes no restrictions on $M$.  It must also be true that $M
\zl_2\in \zL_1$.  This inclusion is equivalent to the congruence
$b\equiv 0\text{ mod }6$.  So if $M\in \zG(2)$, then $M$
stabilizes $\zL_1$ if and only if $b\equiv 0\text{ mod }6$.

In addition to the condition that $M$ stabilizes $\zL_1$, we
need that the map which $M$ induces on $\mathcal{A}=\zL_2/2\zL_1$
fixes the elements of our Hurwitz structure set $\cH \cS=\{\pm
(1,0),\pm (2,0),\pm (1,1),\pm (2,1)\}$ up to $\pm 1$.  An easy
computation shows that $M$ takes $(1,0)$ and $(1,1)$ to themselves up
to $\pm 1$ if and only if $a\equiv \pm 1\text{ mod }12$ and $b\equiv
0\text{ mod }12$.  The corresponding conditions for $(2,0)$ and
$(2,1)$ are then also satisfied.  So $G_f$ contains the images in
$\text{PSL}(2,\bbZ)$ of all such matrices $M$ in $\text{SL}(2,\bbZ)$
such that $a\equiv \pm 1\text{ mod } 12$, $b\equiv 0\text{ mod } 12$,
$c\equiv 0\text{ mod } 2$ and $d\equiv a \text{ mod } 12$.  To compute
$G_f$, in addition to considering multiplication by such matrices $M$,
we must also consider multiplication by such matrices $M$ together
with the translation by a vector $b$ in $2\zL_2\cap
\{0,\zl_1,\zl_2,\zl_1+\zl_2\}=\{0,\zL_1\}$.  We find that this adds no
more elements to $G_f$.  We conclude that $G_f$ is the image in
$\text{PSL}(2,\bbZ)$ of all matrices $M=\left[\begin{smallmatrix}a & b
\\ c & d\end{smallmatrix}\right]\in \text{SL}(2,\bbZ)$ such that
$a\equiv \pm 1\text{ mod } 12$, $b\equiv 0\text{ mod } 12$, $c\equiv
0\text{ mod } 2$ and $d\equiv a \text{ mod } 12$.

The simplest nontrivial element $\zv$ of $G_f$ is represented by the
matrix $\left[\begin{smallmatrix}1 & 0 \\ 2 & 1
\end{smallmatrix}\right]$.  We now compute the image $\widetilde{\zv}$ of
$\zv$ under the modular group virtual endomorphism.

A horizontal line in $\bbR_1^2$ which avoids $\bbZ^2$ maps to a simple
closed curve in $S_1^2$ which separates two points of $P_2$ from the
other two points of $P_2$.  So the slope function $\zm_f$ of $f$ maps
slope 0 to a slope.  To evaluate $\zm_f(0)$, consider
Figure~\ref{fig:CmpuMu}.  The line segment $a$ has ordinary slope 0,
Its endpoints map to $P_2$ and no element of its interior maps to
$P_1\cup P_2$.  Applying Theorem 5.3 of \cite{cfpp} with the line
segment $S$ there equal to $a$, we find that $\zm_f(0)$ is the slope
relative to the basis $(\zl_1,\zl_2)$ of the line segment joining the
centers of the spin mirrors whose endpoints are the endpoints of $a$.
The slope relative to $(\zl_1,\zl_2)$ of the line segment joining
$(0,-1)$ and $(6,1)$ is 2, so $\zm_f(0)=2$.

We compute $\zm_f(\infty )$ in the same way using the line segment
$b$.  The result is the slope relative to $(\zl_1,\zl_2)$ of the line
segment joining $(0,0)$ and $(6,1)$, namely, 1.  We have the top of 
Figure~\ref{fig:MuDgm}.

As in Section 6 of \cite{fkklpps}, we find that the pullback map
$\zm_\zv$ on slopes induced by $\zv$ is given by the matrix
$\left[\begin{smallmatrix}1 & -2 \\ 0 & 1 \end{smallmatrix}\right]$.
This yields the left side of Figure~\ref{fig:MuDgm}.

In order to establish the bottom of Figure~\ref{fig:MuDgm}, we must
evaluate $\zm_f(-2)$.  For this we use line segment $c$ in
Figure~\ref{fig:CmpuMu}.  This computation is a bit more complicated
than for $a$ and $b$ because $c$ crosses two spin mirrors.  We
traverse $c$ starting at $(2,0)$.  We meet the spin mirror centered at
$(6,1)$ and spin to line segment $c'$.  We eventually meet the spin
mirror centered at $(12,1)$ and spin to line segment $c''$.  We
conclude that $\zm_f(-2)$ is the slope relative to $(\zl_1,\zl_2)$ of
the line segment joining $(0,-1)$ and $(18,3)$.  Hence
$\zm_f(-2)=\frac{4}{3}$.  We now have the bottom of
Figure~\ref{fig:MuDgm}. 

Because Figure~\ref{fig:MuDgm} is commutative, the matrix representing
$\zm_{\widetilde{\zv}}$ is 
  \begin{equation*}
\left[\begin{matrix}1 & 4 \\ 1 & 3 \end{matrix}\right]
\left[\begin{matrix}1 & 2 \\ 1 & 1 \end{matrix}\right]^{-1}=
\left[\begin{matrix}1 & 4 \\ 1 & 3 \end{matrix}\right]
\left[\begin{matrix}-1 & 2 \\ 1 & -1 \end{matrix}\right]=
\left[\begin{matrix}3 & -2 \\ 2 & -1 \end{matrix}\right].
  \end{equation*}
Because the upper right and lower left entries of this matrix are
additive inverses of each other, this matrix also represents
$\widetilde{\zv}$.

According to {\tt NETmap}, the degree of the map $X$ is 6 and the
degree of the map $Y$ is 16. Finding the algebraic curve defining the
correspondence on moduli space discussed in the introduction will be difficult.

  \begin{figure}
\centerline{\includegraphics{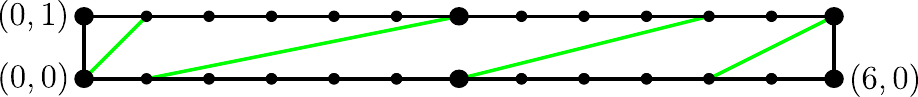}}
\caption{A virtual presentation diagram for $f$}
\label{fig:PrenDgm}
  \end{figure}

  \begin{figure}
\centerline{\includegraphics{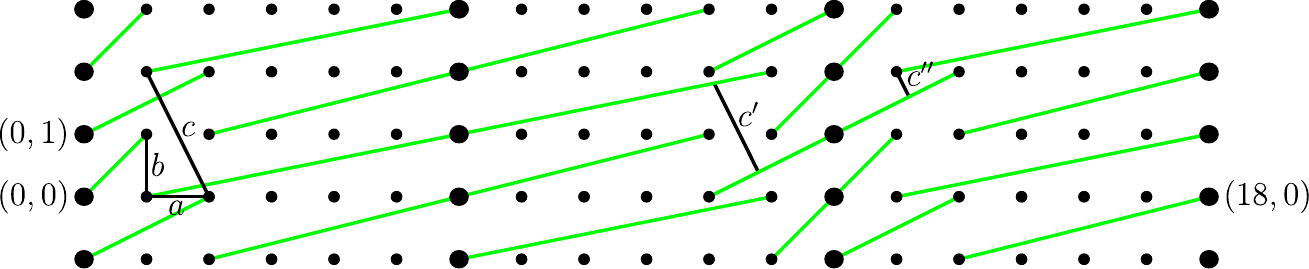}}
\caption{Computing $\zm_f$}
\label{fig:CmpuMu}
  \end{figure}
  
  \newpage

\begin{figure}
  \begin{equation*}
\xymatrix{\frac{1}{0},\frac{0}{1}\ar[d]_{\zm_{\zv}}
\ar[r]^{\zm_f} & \frac{1}{1},\frac{2}{1}\ar[d]_{\zm_{\widetilde{\zv}}}\\
\frac{1}{0},\frac{-2}{1}\ar[r]^{\zm_f} & \frac{1}{1},\frac{4}{3}}
  \end{equation*}
\caption{Computing $\zm_{\widetilde{\zv}}$.}
\label{fig:MuDgm}
\end{figure}

\section{Branch data}\label{sec:bd}\nosubsections

Suppose $X, Y$ are closed oriented surfaces and $f: X \to Y$ is a
finite orientation-preserving branched covering of degree $d \geq 2$.
Let $C_f$ be the set of critical points of $f$, and let
$V_f=f(C_f)$, the set of critical values of $f$.  Generalizing the
definition used in the introduction and \S 6, we say another such map
$f': X' \to Y'$ is \emph{topologically equivalent} to $f: X \to Y$
if there exist orientation-preserving homeomorphisms $\phi: X \to X',
\psi: Y \to Y'$ with $\psi \circ f = f' \circ \phi$.  Given $y \in Y$,
the collection $\{ \deg(f,x) : f(x)=y\}$ of local degrees defines a
partition of $d$. The \emph{branch data} associated to $f$ is the
set of such partitions for $y\in V_f$.  Branch data is an invariant
of topological equivalence.
  
Taking $X=Y=S^2$ and $f$ to be a NET map, the branch data of $f$
yields an invariant of the topological equivalence class of $f$.  In
this section, we study branch data associated to NET maps. Theorem
\ref{thm:eleydivrs} implies that the topological equivalence class of
a NET map is determined by its elementary divisors.  So the branch
data of a NET map is computable from its elementary divisors.  This
will be made explicit after the proof of Lemma~\ref{lemma:ed_and_bd}.

Since each critical point of a NET map has local degree $2$, there is
much redundancy in its branch data.  It is therefore advantageous to
focus on branch value preimages which are not critical points.  From
\cite[Lemma 1.3]{cfpp}, we know that the set $P_1\subseteq S^2-C_f$ of
all noncritical points which $f$ maps to its postcritical set has
exactly four points.  Let $P_1^*=P_1\cap f^{-1}(V_f)$.  The
\emph{kernel} of the branch data of $f$ is the set-theoretic data of
the restriction $f|P_1^*: P_1^* \to f(P_1^*)$, regarded as a function
from one finite set to another.

\begin{lemma} \label{lemma:ed_and_bd} Suppose $f$ is a NET map. Let
$(m,n)$ with $n|m$ be the elementary divisors of $f$. Exactly one of
the following three cases occurs.
\begin{enumerate}
\item $(m,n) \equiv (1,1)$ modulo 2 $\iff |V_f-f(P_1^*)|=0$
\item $(m,n) \equiv (0,1)$ modulo 2 $\iff |V_f-f(P_1^*)|=2$
\item $(m,n) \equiv (0,0)$ modulo 2 $\iff |V_f-f(P_1^*)|=3$
\end{enumerate}
\end{lemma}

\begin{proof} We first observe that $V_f-f(P_1^*)=P_2-f(P_1)$, where
$P_2$ is the postcritical set of $f$.  Up to topological equivalence,
$f$ is the map on the sphere induced by the diagonal matrix $A$ with
diagonal entries $(m,n)$. By Theorem~\ref{thm:eleydivrs}, we may
assume that $f$ has this form.  The domain sphere is the double of an
$m$-by-$n$ rectangle over its boundary; the codomain sphere is the
double of the unit square over its boundary. The map $f$ folds in the
obvious way. The set $P_1$ consists of the corners of the domain
pillowcase.  From this the lemma follows immediately.

A bit more formally: the set of corners $P_1$ of the domain sphere is
naturally identified with the vector space $\Lambda_1/2\Lambda_1$;
that of the codomain sphere with the vector space
$\Lambda_2/2\Lambda_2$; the folding map induces a linear map between
these vector spaces. Cases 1, 2, 3 then correspond respectively to the
rank of this linear map being 2, 1, or 0.
\end{proof}

The branch data of $f$ is easily reconstructed from knowledge of its
degree $d$ and its branch data kernel: the kernel specifies precisely
which partitions have a $1$ as an additive factor, and how many times
a $1$ occurs.  Here is the precise statement. In each partition, the
number of $2$'s is uniquely determined by the condition that the sum
is the degree, $d$.
\begin{enumerate} 
  \item $|V_f-f(P_1^*)|=0$; the partition over each element of $V_f$
is $\{2, \ldots, 2, 1\}$
  \item $|V_f-f(P_1^*)|=2$; the  partitions which contain 1's have the
form $\{2, \ldots, 2, 1, 1\}$
  \item $|V_f-f(P_1^*)|=3$; the partition which contains 1's has the
form $\{2, \ldots, 2, 1,1,1,1\}$
\end{enumerate}
In the typical case, exactly two partitions contain 1's in (2)
and exactly one partition contains 1's in (3).  In the atypical case
fewer partitions may contain 1's.  Below, we refer to the three types
of branch data by the same case numbers.

The question of which branch data arise from branched coverings--that
is, are \emph{realizable}--is a classical, much-studied problem; see
\cite{pp} and the references therein. The Riemann-Hurwitz formula
gives an obvious necessary condition: here, this means that there are
$2d-2$ critical points, counted with multiplicity.  We always assume
that branch data satisfy this condition.  The Riemann-Hurwitz
conditions are not sufficient for branch data to be realizable. For
example, the degree 6 branch data $\{2, 1,1,1,1\}, 3 \times \{2,2,2\}$
is not realizable because statement 3 of Lemma~\ref{lemma:ed_and_bd}
implies that 4 divides $d$ when $|V_f-f(P_1^*)|=3$. In general,
\cite[Thm. 3.8]{pp} implies the following.

\begin{thm} \label{thm:pp3point8} Branch data of types 1 and 2 are
always realizable. Branch data of type 3 is realizable if and only if
$d \equiv 0$ modulo $4$. In each realizable case, the realizing map
can be taken to be a Euclidean NET map.
\end{thm}
Note that the extra condition in the hypothesis in Case 3 is indeed
satisfied by branch data of such NET maps (Lemma \ref{lemma:ed_and_bd}).

\section{Dynamic portraits}
\label{sec:main}\nosubsections

This section deals with dynamic portraits of NET maps.
Corollary~\ref{cor:portrait} completely characterizes dynamic
portraits of NET maps.  Theorem~\ref{thm:portrait} gives, in addition,
information about their relationship with elementary divisors.
Algorithms for computing the correspondence between dynamic portraits
of NET maps and NET map presentations are presented in the next two
sections.  Table~\ref{tab:portraitnum} gives the number of NET map
dynamic portraits in every degree.

We begin by discussing dynamic portraits of general Thurston
maps from here to Lemma~\ref{lemma:closing}.  By a \emph{dynamic
portrait} of a Thurston map $f$ we mean a weighted directed graph
$\zG$ whose vertices are in bijective correspondence with the points
of $S^2$ which are either critical or postcritical points of $f$.
There is a directed edge from a vertex of $\zG$ corresponding to $x\in
S^2$ to a vertex of $\zG$ corresponding to $y\in S^2$ if and only if
$y=f(x)$.  This edge is assigned a weight, the positive integer
$\deg(f,x)$.

Let $\zG$ be a finite directed graph whose edges are weighted with
positive integers for which exactly one edge emanates from every
vertex.  We define critical points (vertices), critical values
(vertices) and postcritical vertices of $\zG$ in the straightforward
way.  We define the incoming degree of a vertex $v$ of $\zG$ to be the
sum of the weights of the edges with terminal vertex $v$.  The
Riemann-Hurwitz formula imposes conditions on the dynamic portrait of
a Thurston map $f$:
  \begin{equation*}
\sum_{x\in C_f}^{}(\deg(f,x)-1)=2d-2,
  \end{equation*}
where $C_f$ is the set of critical points of $f$ and $d=\deg(f)$.  In
the same way, we say that $\zG$ satisfies the Riemann-Hurwitz formula
if there exists an integer $d\ge 2$ for which the analogous equation
holds and the incoming degree of every vertex of $\zG$ is at most $d$.
We say that a finite weighted directed graph $\zG$ satisfies the
\emph{obvious conditions} for the dynamic portrait of a Thurston map if
exactly one edge emanates from every vertex, each of its vertices is
either critical or postcritical and it satisfies the Riemann-Hurwitz
condition for some integer $d\ge 2$, the degree of $\zG$.

Let $\zG$ be a finite weighted directed graph which satisfies the
obvious conditions for the dynamic portrait of a Thurston map.  We can
assign branch data to $\zG$ just as we assign branch data to a
Thurston map.  We are interested in whether or not $\zG$ is
realizable, that is, isomorphic to the dynamic portrait of a Thurston
map.  It is clear that if $\zG$ is realizable, then its branch data is
realizable.  The next lemma reduces the realizability of $\zG$ to the
realizability of its branch data.

\begin{lemma} \label{lemma:closing} Let $\zG$ be a finite weighted
directed graph which satisfies the obvious conditions for the dynamic
portrait of a Thurston map.  Suppose that the branch data of $\zG$ is
realizable by a Thurston map $f$.  Then $\zG$ is realizable by a
Thurston maps topologically equivalent to $f$.
\end{lemma}

\begin{proof} Let $A$ be the set of vertices of $\zG$ and let $B$ be
the set of vertices of $\zG$ with at least one incoming edge.  We
regard $A, B$ as disjoint sets.  The edges of $\zG$ determine a
surjective map $\phi: A \to B$.  The identity map defines an injection
$\iota: B \hookrightarrow A$.  Denote by $C \subset A$ the set of
critical points and $V:=\phi(C)\subset B$ the set of critical values.
Let $S^2_A, S^2_B$ be distinct copies of the sphere. Suppose the
branch data extracted from $\zG$ is realizable.  By hypothesis there
exist a Thurston map $f\co S_A^2\to S_B^2$ and embeddings $C
\hookrightarrow S^2_A$ and $V \hookrightarrow S^2_B$ such that, upon
identifying elements of $C$ and of $V$ with their images in the
respective spheres, we have $\phi|_C=f|_C$, and the local degrees at
points in $C$ coming from the portrait and from $f$ coincide.  We
extend the embedding $V \hookrightarrow S^2_B$ arbitrarily to an
embedding $B \hookrightarrow S^2_B$.

In this paragraph, we show that there is an extension of the embedding
$C \hookrightarrow S^2_A$ to an embedding $A \hookrightarrow S^2_A$
such that $\phi|_A=f|_A$. Pick $b \in B$. If $\phi^{-1}(b) \cap (A-C)
= \emptyset$, then there is nothing to do. Otherwise, let $\{a_1, \ldots,
a_k\}=\phi^{-1}(b) \cap (A-C)$.  The local degree at each $a_i$ is
equal to $1$. By the local degree constraint on $\zG$ 
and the compatibility of both local and global degrees, $f^{-1}(b)
\cap (S^2_A-C)$ contains at least $k$ elements. We choose an embedding
$\{a_1, \ldots, a_k\} \hookrightarrow f^{-1}(b) \cap (S^2_A-C)$ (it is
not necessarily surjective: the sum of the local degrees in every
fiber of $\zf$ is at most but not necessarily equal to $\deg(f)$).
Repeating this construction for each $b$ establishes the result.

Finally, we ``close'' the map $f: A \to B$ to obtain a realization of
the dynamic portrait. Recall that we have an injection $\iota: B \to
A$ induced by the identity map on the vertices in the portrait. We
arbitrarily choose an orientation-preserving homeomorphism $h: S^2_B
\to S^2_A$ which agrees with $\iota$. The composition $h \circ f$ is a
Thurston map topologically equivalent to $f$ which realizes $\zG$.
\end{proof}

Now we specialize to NET maps.  A finite weighted directed graph $\zG$
satisfies the obvious conditions for the dynamic portrait of a NET map
if it satisfies the obvious conditions with some additional restrictions: every edge weight is either 1 or 2, and the postcritical set $P$ of $\zG$ has exactly four points.  

Let $\zG$ be a finite weighted directed graph which satisfies the
obvious conditions for the dynamic portrait of a NET map.  In this
paragraph we define the \emph{mod 2 elementary divisors} $m,n\in
\{0,1\}$ of $\zG$.  Motivated by Lemma~\ref{lemma:ed_and_bd}, let $V$
be the set of critical values of $\zG$.  Let $k$ be the number of
elements $x$ of $V$ whose incoming degree is the degree of $\zG$ and
no edge of $\zG$ ending at $x$ has weight 1.  As in
Lemma~\ref{lemma:ed_and_bd}, we have that $k$ is either 0, 2 or 3.
The mod 2 elementary divisors $(m,n)$ of $\zG$ are defined to be
either $(1,1)$, $(0,1)$ or $(0,0)$ accordingly.

We continue to let $\zG$ be a finite weighted directed graph which
satisfies the obvious conditions for the dynamic portrait of a NET map
with degree d.  We say that $\zG$ satisfies the \emph{exceptional
condition} in the following situation.  If the mod 2 elementary
divisors of $\zG$ are both 0, then 4 divides $d$.

We are finally ready for the main theorem on dynamic portraits of NET
maps.

\begin{thm}\label{thm:portrait} Let $\zG$ be a finite weighted
directed graph which satisfies the obvious conditions for the dynamic
portrait of a NET map with degree d.  Let $m$ and $n$ be positive
integers such that $n|m$ and $mn=d$.  Then there exists a NET map with
a dynamic portrait isomorphic to $\zG$ and elementary divisors $m$ and
$n$ if and only if $m$ and $n$ are congruent modulo 2 to the mod 2
elementary divisors of $\zG$.
\end{thm}

\begin{cor}\label{cor:portrait} Let $\zG$ be a finite weighted
directed graph which satisfies the obvious conditions for the dynamic
portrait of a NET map.  Then $\zG$ is isomorphic to the dynamic
portrait of a NET map if and only if $\zG$ satisfies the exceptional
condition.
\end{cor}
\begin{proof} The point here is that if $\zG$ satisfies the
exceptional condition, then there exist positive integers $m$ and $n$
congruent modulo 2 to the mod 2 elementary divisors of $\zG$ such that
$n|m$ and $mn=d$.
\end{proof}

\begin{remark}\label{remark:portrait} Pascali and Petronio prove
Corollary~\ref{cor:portrait} for Euclidean NET maps in Theorem 3.8 of
\cite{pp}.  Our Theorem~\ref{thm:portrait} gives additional
information about elementary divisors.  It implies that in general
such a graph is isomorphic to the dynamic portrait of many NET maps
which are not only mutually Thurston inequivalent but even
topologically inequivalent.  Using the algorithm in
Section~\ref{sec:maptopic}, it is possible to find a NET map with a
given dynamic portrait.  This is done in \cite{NET} for every dynamic
portrait through degree 40.
\end{remark}

\begin{remark}\label{remark:portrait} It is not difficult to see that
there is a global bound on the number of NET map dynamic portraits for
a given degree.  Table~\ref{tab:portraitnum} gives the number $n$ of
NET map dynamic portraits in each degree $d$.
\end{remark}

\begin{table}
\begin{center}
\begin{tabular}{r|cccccccccc}
$d$  &  2  &  3  &  4  &  5  &  6  &  7  &  8  \\ \hline
$n$  & 16 & 94 & 272 & 144 & 338 & 152 & 476  \\  \hline\hline
$d \mod 4, d \geq 9$      &   0   &   1  &   2     &    3   \\ \hline
$n$  & 483 & 153 & 353 & 153 \\ 
\end{tabular}
\end{center}
\caption{The number $n$ of dynamic portraits for NET maps of degree $d$}
\label{tab:portraitnum} \end{table}

\begin{proof}[Proof of Theorem~\ref{thm:portrait}] To prove the
forward implication, suppose that $f$ is a NET map that realizes
$\Gamma$ and has elementary divisors $(m,n)$.  Then $f$ realizes the
branch data of $\Gamma$.  By Lemma \ref{lemma:ed_and_bd}, the mod 2
elementary divisors of $\Gamma$ are congruent to $(m,n)$ modulo $2$.

To prove the backward implication, suppose that the mod 2 elementary
divisors of $\Gamma$ are congruent modulo 2 to $(m,n)$.  Let $g$ be a
Euclidean NET map induced by a diagonal matrix with diagonal entries
$m$ and $n$.  The elementary divisors of $g$ are $m$ and $n$.
Lemma~\ref{lemma:ed_and_bd} and the discussion following it show that
the branch data of $g$ is determined by the modulo 2 congruence
classes of $m$ and $n$.  The branch data of $\zG$ is determined by its
mod 2 elementary divisors in the same way.  Hence the branch data of
$\zG$ equals the branch data of $g$.  Now Lemma~\ref{lemma:closing}
and Theorem~\ref{thm:eleydivrs} complete the proof of
Theorem~\ref{thm:portrait}.
\end{proof}

\begin{remark}\label{remark:hyperbolictype} A dynamic portrait has
\emph{hyperbolic type} if each cycle contains a critical point. For
topological polynomials, any Thurston map whose dynamic portrait has
hyperbolic type is unobstructed, i.e. equivalent to a complex
polynomial \cite[Theorem 10.3.9]{h2}.  Kelsey \cite{K} establishes a
partial converse: for almost all polynomial dynamic portraits of
non-hyperbolic type, there is an obstructed topological polynomial
with this portrait.  Our data show this association between portraits
and obstructedness breaks down for NET maps: the modular Hurwitz class
31HClass7 (and many others) consists entirely of unobstructed maps,
and there are non-critical postcritical points. Composing with
translations, we may find within this class examples of unobstructed
NET maps with a non-critical fixed postcritical point.
\end{remark}

\section{Explicit construction of the dynamic portrait of a NET
map}\label{sec:maptopic}\nosubsections

In this section we present an algorithm which constructs the dynamic
portrait of a NET map from a presentation for it.  Starting with a
presentation, we determine what to do, and then we give a step-by-step
algorithm.  We conclude with an example.

Let $(A,b,\za_1,\za_2,\za_3,\za_4)$ be presentation data for a NET map
$f$.  As usual, this allows $f$ to be expressed as a composition
$f=h\circ g$.  We seek combinatorial analogs $\zf$, $\zg$ and $\zh$
for $f$, $g$ and $h$.  We will use these combinatorial analogs to
construct a dynamic portrait for $f$.

To begin, let $\zl_1$ and $\zl_2$ be the columns of $A$.  They
generate a sublattice $\zL_1$ of $\zL_2=\bbZ^2$ with index
$\text{det}(A)=\deg(f)$.  Each green line segment $\za_i$ contains
exactly one element of $\zL_1$.  This element of $\zL_1$ is an
endpoint of $\za_i$.  We view $\za_i$ as directed, with this endpoint
its initial endpoint.  The four initial endpoints of $\za_1$, $\za_2$,
$\za_3$, $\za_4$ map bijectively to $\zL_1/2\zL_1$.  Let
$P_1=\zL_1/2\zL_1$.  The set $P_1$ will contain the noncritical points
of $\zG$.  Let $P_2$ be the set of images in $(\zL_2/2\zL_1)/\{\pm
1\}$ of the terminal endpoints of $\za_1$, $\za_2$, $\za_3$, $\za_4$.
The set $P_2$ will be the postcritical set of $\zG$.

Now we define $\zh\co P_1\to P_2$.  Let $x\in P_1$.  There exists a
unique $i\in \{1,2,3,4\}$ such that $x$ is the image in $P_1$ of the
initial endpoint of $\za_i$.  Let $\zh(x)$ be the image in $P_2$ of
the terminal endpoint of $\za_i$.

We define $\zg\co P_1\cup P_2\to P_1$ in this paragraph.  There exist
integers $b_1$ and $b_2$ such that $b=b_1\zl_1+b_2\zl_2$.  For every
$\zl\in \zL_1$, let $\overline{\zl}$ be the image of $\zl$ in
$\zL_1/2\zL_1$.  To define $\zg$, let $x\in P_1\cup P_2$.  Let $(r,s)$
be a representative of $x$ in $\bbZ^2$.  Set
  \begin{equation*}
\zg(x)=(r+b_1)\overline{\zl}_1+(s+b_2)\overline{\zl}_2\in P_1.
  \end{equation*}

Now that we have $\zg$ and $\zh$, we define $\zf\co P_1\cup P_2\to
P_2$ so that $\zf=\zh\circ \zg$.  We can easily construct a dynamic
portrait for $f$ from $\zf$ and $d=\text{det}(A)=\deg(f)$.

Here is a step-by-step algorithm based on the above discussion for
constructing a dynamic portrait for a NET map with presentation data
$(A,b,\za_1,\za_2,\za_3,\za_4)$.
\medskip

\noindent\textsl{Step 1} Let $\zl_1$ and $\zl_2$ be the columns of
$A$, let $\zL_2=\bbZ^2$ and let $\zL_1$ be the sublattice of $\zL_2$
generated by $\zl_1$ and $\zl_2$.  Direct $\za_i$ so that the initial
endpoint of $\za_i$ is contained in $\zL_1$ for $i\in \{1,2,3,4\}$.
Let $\overline{\zl}$ denote the image in $\zL_1/2\zL_1$ of $\zl$ in
$\zL_1$.
\medskip

\noindent\textsl{Step 2} Let $P_1=\zL_1/2\zL_1$ be the set of images
in $(\zL_2/2\zL_1)/\{\pm 1\}$ of the initial endpoints of $\za_1$,
$\za_2$, $\za_3$, $\za_4$, and let $P_2$ be the set of images in
$(\zL_2/2\zL_1)/\{\pm 1\}$ of the terminal endpoints of $\za_1$,
$\za_2$, $\za_3$, $\za_4$.
\medskip

\noindent\textsl{Step 3} Define $\zh\co P_1\to P_2$ so that $\zh$ maps
the image in $P_1$ of the initial endpoint of $\za_i$ to the image in
$P_2$ of the terminal endpoint of $\za_i$ for $i\in \{1,2,3,4\}$.
\medskip

\noindent\textsl{Step 4} Find the integers $b_1$ and $b_2$ such that
$b=b_1\zl_1+b_2\zl_2$.
\medskip

\noindent\textsl{Step 5} For every $x\in P_1\cup P_2$ choose a
representive $(r,s)$ for $x$ in $\bbZ^2$, and compute 
  \begin{equation*}
\zg(x)=(r+b_1)\overline{\zl}_2+(s+b_2)\overline{\zl}_2\in
\zL_1/2\zL_1=P_1.
  \end{equation*}

\noindent\textsl{Step 6} For every $x\in P_2$ construct an edge with
initial vertex $x$ and terminal vertex $\zf(x)=\zh(\zg(x))$.  This
edge is given weight 1 if $x\in P_1$ and weight 2 if $x\in
P_2\setminus P_1$.
\medskip

\noindent\textsl{Step 7} Adjoin enough new edges with weight 2,
distinct initial endpoints not in $P_1\cup P_2$ and terminal endpoints
in $P_2$ so that the incoming degree of every element $x$ of $P_2$
plus the order of $\zf^{-1}(x)\cap (P_1-P_2)$ is $d=\text{det}(A)$.
\medskip

Now we have our dynamic portrait.

\begin{ex}\label{ex:first} We illustrate this construction for the
degree 10 NET map considered in Example 3.1 of \cite{cfpp}.
Figure~\ref{fig:firstprendgm} has a presentation diagram for it.  We
write rows instead of columns to simplify notation.

  \begin{figure}
\centerline{\includegraphics{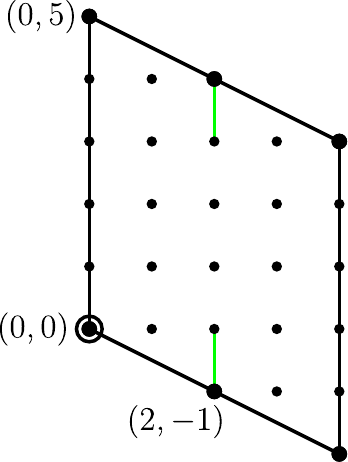}}
\caption{A presentation diagram for Example~\ref{ex:first} }
\label{fig:firstprendgm}
  \end{figure}

We have that $A=\left[\begin{smallmatrix}2 & 0 \\ -1 & 5
\end{smallmatrix}\right]$.  The circled lattice point in
Figure~\ref{fig:firstprendgm} is $(0,0)$, so $b=(0,0)$.  The green
line segments $\za_1$ and $\za_3$ are trivial, consisting of the
points $(0,0)$ and $(0,5)$.  The green line segment $\za_2$ joins
$(2,-1)$ and $(2,0)$.  The green line segment $\za_4$ joins $(2,4)$
and $(2,3)$.
\medskip

\noindent\textsl{Step 1} $\zl_1=(2,-1)$ \quad $\zl_2=(0,5)$

\noindent\textsl{Step 2} 
$P_1=\{\overline{(0,0)},\overline{(2,-1)},\overline{(0,5)},
\overline{(2,4)}\}$ \quad
$P_2=\{\overline{(0,0)},\overline{(2,0)},\overline{(0,5)},
\overline{(2,3)}\}$
\medskip

\noindent\textsl{Step 3} $\zh(\overline{(0,0)})=\overline{(0,0)}$
\quad $\zh(\overline{(2,-1)})=\overline{(2,0)}$
\quad $\zh(\overline{(0,5)})=\overline{(0,5)}$
\quad $\zh(\overline{(2,4)})=\overline{(2,3)}$
\medskip

\noindent\textsl{Step 4} $b_1=0$ \quad $b_2=0$ 
\medskip

\noindent\textsl{Step 5} We compute as follows.
  \begin{alignat*}{5}
\zf(\overline{(0,0)}) &=& \zh(\zg(\overline{(0,0)}))
&=& \zh((0+0)\overline{\zl}_1+(0+0)\overline{\zl}_2) 
&=& \zh(0) &=& \overline{(0,0)}
\\
\zf(\overline{(2,-1)}) &=& \hspace{.3em}\zh(\zg(\overline{(2,-1)}))
&=& \hspace{.3em}\zh((2+0)\overline{\zl}_1+(-1+0)\overline{\zl}_2)
&=&\hspace{.3em}\zh(\overline{\zl}_2) &=& \hspace{.3em}\overline{(0,5)}
\\
\zf(\overline{(0,5)}) &=& \zh(\zg(\overline{(0,5)}))
&=& \zh((0+0)\overline{\zl}_1+(5+0)\overline{\zl}_2) 
&=& \zh(\overline{\zl}_2) &=& \overline{(0,5)}
\\
\zf(\overline{(2,4)}) &=& \zh(\zg(\overline{(2,4)}))
&=& \zh((2+0)\overline{\zl}_1+(4+0)\overline{\zl}_2)
&=& \zh(0) &=& \overline{(0,0)}
\\
\zf(\overline{(2,0)}) &=& \zh(\zg(\overline{(2,0)}))
&=& \zh((2+0)\overline{\zl}_1+(0+0)\overline{\zl}_2)
&=& \zh(0) &=& \overline{(0,0)}
\\
\zf(\overline{(2,3)}) &=& \zh(\zg(\overline{(2,3)}))
&=& \zh((2+0)\overline{\zl}_1+(3+0)\overline{\zl}_2) 
&=& \zh(\overline{\zl}_2) &=& \overline{(0,5)}
  \end{alignat*}

\noindent\textsl{Step 6} We construct a graph as in
Figure~\ref{fig:firstkernel}. 
\medskip

\noindent\textsl{Step 7} We adjoin $16$ critical vertices and
corresponding edges to obtain Figure~\ref{fig:firstfinal}.  This is
essentially Figure 6 in \cite{cfpp}.

  \begin{figure}
\centerline{\includegraphics{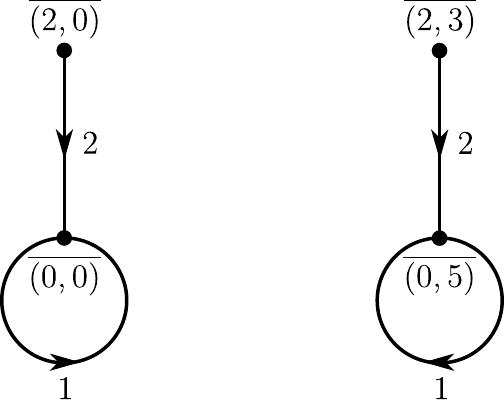}}
\caption{Step 6 for Example~\ref{ex:first}}
\label{fig:firstkernel}
  \end{figure}

  \begin{figure}
\centerline{\includegraphics{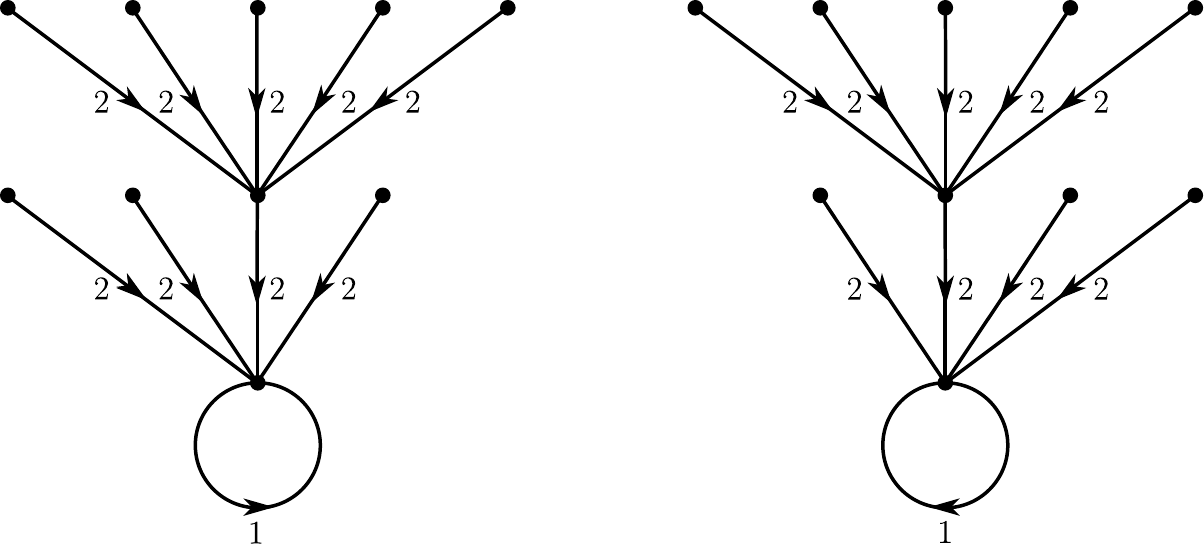}}
\caption{The dynamic portrait for Example~\ref{ex:first}}
\label{fig:firstfinal}
  \end{figure}
\end{ex}

\section{Explicit construction of NET maps from a dynamic
portrait}\label{sec:pictomap}\nosubsections

In this section we begin with a weighted directed graph $\zG$ which
satisfies the obvious conditions for the dynamic portrait of a NET map
and also the exceptional condition.  We present an algorithm which
constructs at least one representative of every Thurston equivalence
class of NET maps with dynamic portrait isomorphic to $\zG$.  This
algorithm is essentially gotten by outlining the proofs of
Lemma~\ref{lemma:ed_and_bd} and the backward implication of
Theorem~\ref{thm:portrait}.  We conclude with an example.
\medskip

\noindent\textsl{Step 1} Determine the degree $d$ of $\zG$.
\medskip

\noindent\textsl{Step 2} Identify the set $P_2$ of postcritical
vertices of $\zG$.  Let $P_1$ be the set of noncritical vertices of
$\zG$ together with more points disjoint from $\zG$ if necessary so
that $|P_1|=4$.  Define a map map $\zf\co P_1\cup P_2\to P_2$ so that
$\zf|_{P_2}$ is the map determined by $\zG$ and $\zf|_{P_1-P_2}$ is
defined in any way so that the incoming degree of every element $x$ of
$P_2$ plus the order of $\zf^{-1}(x)\cap (P_1-P_2)$ is $d$.
\medskip

\noindent\textsl{Step 3} Choose positive integers $m$ and $n$ such
that $n|m$, $mn=d$ and $n$ is even if and only if the restriction of
$\zf$ to $P_1$ is constant.
\medskip

\noindent\textsl{Step 4} Let $\zL_2=\mathbb{Z}^2$ and
$\zL_1'=\left<(m,0),(0,n)\right>$.  Let $\zG'_1$ be the group of
Euclidean isometries of the form $x\mapsto 2\zl\pm x$ for some $\zl\in
\zL'_1$.  Let $F'_1$ be the rectangle in $\mathbb{R}^2$ with corners
$(0,0)$, $(2m,0)$, $(0,n)$ and $(2m,n)$.  Identify $P_1\cup P_2$ with
a subset of $(\zL_2/2\zL'_1)/\{\pm 1\}$ by labeling appropriate 
elements of $\zL_2\cap F'_1$ with the elements of $P_1\cup P_2$.  Then
construct four directed line segments $\za'_1$, $\za'_2$, $\za'_3$,
$\za'_4$ in $F'_1$ whose images in $\mathbb{R}^2/\zG'_1$ are disjoint
such that
\begin{enumerate}
  \item  $P_1=\{\overline{(0,0)},\overline{(m,0)},\overline{(0,n)},
\overline{(m,n)}\}$;
  \item $\zf(x)=\zf(y)$ if and only if $x\equiv y \text{ mod } 2
\zL_2$ for $x,y\in P_1\cup P_2$;
  \item the initial endpoint of $\za'_i$ maps to $P_1$ and the
terminal endpoint of $\za'_i$ maps to $P_2$ for every $i\in
\{1,2,3,4\}$.
\end{enumerate}
Let $\zh\co P_1\to P_2$ be the bijection which takes the image of the
initial endpoint of $\za'_i$ to the image of the terminal endpoint of
$\za'_i$ for every $i\in \{1,2,3,4\}$.
\medskip

\noindent\textsl{Step 5} The map $\zh^{-1}\circ \zf$ determines an
injective function from a subset of $\zL_2/2\zL_2$ to $\zL'_1/2
\zL'_1$.  This injection extends to a bijection $\zq\co
\zL_2/2\zL_2\to \zL'_1/2\zL'_1$.  There are 24 such bijections because
$|\zL_2/2\zL_2|=|\zL'_1/2\zL'_1|=4$.  Letting $\bbF_2$ denote the
field of two elements, $|\text{SL}(2,\bbF_2)|=6$ and $|\bbF_2^2|=4$.
Hence there are $6\times 4=24$ affine isomorphisms between two
2-dimensional vector spaces over $\bbF_2$.  Thus $\zq$ is an affine
isomorphism.  Find $\overline{Q}\in \text{SL}(2,\mathbb{F}_2)$ and
$\overline{b}_0\in \mathbb{F}_2^2$ such that $\zq$ is given by
$x\mapsto \overline{Q}x+\overline{b}_0$ with respect to the bases of
$\zL_2/2\zL_2$ and $\zL'_1/2\zL'_1$ determined by the standard basis
elements $e_1$, $e_2$ of $\zL_2$ and $me_1$, $ne_2$.
\medskip

\noindent\textsl{Step 6} Choose $Q\in \text{SL}(2,\mathbb{Z})$ and
$b_0\in \{0,e_1,e_2,e_1+e_2\}$ such that $\overline{Q}$ and
$\overline{b}_0$ are the reductions of $Q$ and $b_0$ modulo 2.
\medskip

\noindent\textsl{Step 7} Set $A=Q\left[\begin{smallmatrix}m & 0 \\ 0 &
n\end{smallmatrix}\right]$ and $b=Ab_0$, where we view $b_0$ as a
column vector.
\medskip

\noindent\textsl{Step 8} Define $F_1$, $\za_1$, $\za_2$, $\za_3$,
$\za_4$ to be the images of $F'_1$, $\za'_1$, $\za'_2$, $\za'_3$,
$\za'_4$ under the map $x\mapsto Qx$.
\medskip

The result is NET map presentation data
$(A,b,\za_1,\za_2,\za_3,\za_4)$  for a NET map whose dynamic portrait
is isomorphic to $\zG$.

\begin{ex}\label{ex:degfour} We illustrate this construction for the
dynamic portrait in Figure~\ref{fig:degfourgraph}.
\medskip

  \begin{figure}
\centerline{\includegraphics{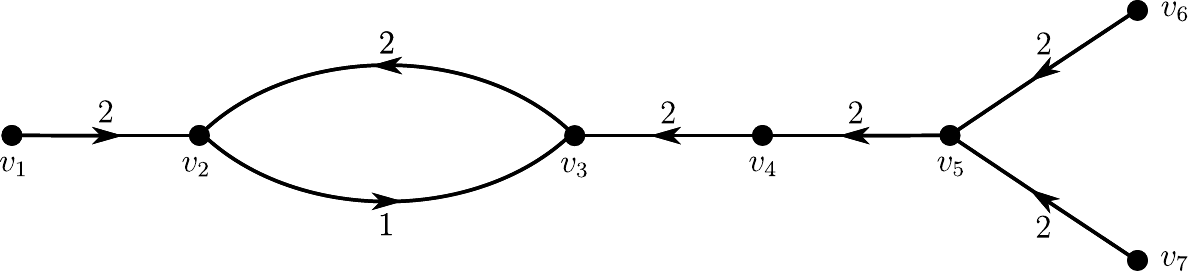}} 
\caption{The dynamic portrait for Example~\ref{ex:degfour}}
\label{fig:degfourgraph}
  \end{figure}

\noindent\textsl{Step 1} $d=4$
\medskip

\noindent\textsl{Step 2} 
$P_1=\{v_2,v_8,v_9,v_{10}\}$, $P_2=\{v_2,v_3,v_4,v_5\}$ \quad
Three additional vertices $v_8$, $v_9$, $v_{10}$ are needed.
  \begin{equation*}
\zf:\quad v_1,v_3\mapsto v_2\quad v_2,v_4,v_8\mapsto v_3\quad
v_5,v_9,v_{10}\mapsto v_4\quad v_6,v_7\mapsto v_5
  \end{equation*}

\noindent\textsl{Step 3} $m=4$, $n=1$
\medskip

\noindent\textsl{Step 4} We label appropriate elements of $\zL_2\cap
F'_1$ with the elements of $P_1\cup P_2$ as indicated in
Figure~\ref{fig:degfourfundom}.  The critical points $v_1$ and $v_3$
are congruent modulo 2.  The critical point $v_4$ and the noncritical
points $v_2$ and $v_8$ are congruent modulo 2.  The critical point
$v_5$ and the noncritical points $v_9$ and $v_{10}$ are congruent
modulo 2.  The critical points $v_6$ and $v_7$ are congruent modulo 2.
The nontrivial green arcs are drawn in Figure~\ref{fig:degfourfundom}.
  \begin{equation*}
\zh\co\quad v_2\mapsto v_2\qquad v_8\mapsto v_3\qquad v_9\mapsto v_4\qquad
v_{10}\mapsto v_5
  \end{equation*} 

  \begin{figure}
\centerline{\includegraphics{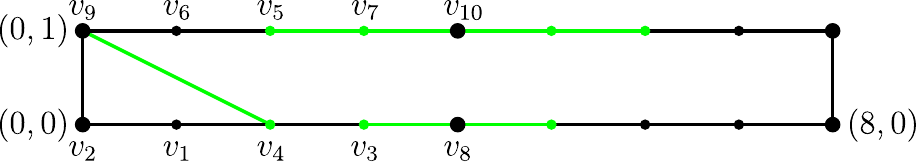}} 
\caption{The preliminary fundamental domain for
Example~\ref{ex:degfour}} \label{fig:degfourfundom}
  \end{figure}

\noindent\textsl{Step 5} We evaluate $\zh^{-1}\circ \zf$ at three
points.
  \begin{align*}
 & \overline{(0,0)}=v_2\overset{\zf}{\longmapsto}v_3
\overset{\zh^{-1}}{\longmapsto} v_8=\overline{(4,0)}=1\cdot
\overline{me_1}+0\cdot \overline{ne_2}\\
 & \overline{(1,0)}= v_1\overset{\zf}{\longmapsto}v_2
\overset{\zh^{-1}}{\longmapsto}v_2=\overline{(0,0)}=0\cdot
\overline{me_1}+0\cdot \overline{ne_2}\\
 & \overline{(0,1)}=v_9\overset{\zf}{\longmapsto}v_4
\overset{\zh^{-1}}{\longmapsto}v_9=\overline{(0,1)}=0\cdot
\overline{me_1}+1\cdot \overline{ne_2}
\end{align*}
Let $\overline{Q}=\left[\begin{matrix}w & x \\ y & z
\end{matrix}\right]\in \text{SL}(2,\mathbb{F}_2)$.
  \begin{equation*}
\begin{aligned}
\therefore\quad  & \left[\begin{matrix}w & x \\ y & z
\end{matrix}\right]
\left[\begin{matrix}0\\0\end{matrix}\right]+\overline{b}_0=
\left[\begin{matrix}1\\ 0\end{matrix}\right]\Longrightarrow 
\overline{b}_0=\left[\begin{matrix}1\\ 0\end{matrix}\right]\\
 & \left[\begin{matrix}w & x \\ y & z\end{matrix}\right]
\left[\begin{matrix}1\\0\end{matrix}\right]+\overline{b}_0=
\left[\begin{matrix}0\\ 0\end{matrix}\right]\Longrightarrow 
\left[\begin{matrix}w\\ y\end{matrix}\right]+
\left[\begin{matrix}1\\ 0\end{matrix}\right]=
\left[\begin{matrix}0\\ 0\end{matrix}\right]\Longrightarrow 
w=1, y=0\\
 & \left[\begin{matrix}w & x \\ y & z\end{matrix}\right]
\left[\begin{matrix}0\\1\end{matrix}\right]+\overline{b}_0=
\left[\begin{matrix}0\\ 1\end{matrix}\right]\Longrightarrow 
\left[\begin{matrix}x\\ z\end{matrix}\right]+
\left[\begin{matrix}1\\ 0\end{matrix}\right]=
\left[\begin{matrix}0\\ 1\end{matrix}\right]\Longrightarrow 
x=1, z=1
\end{aligned}
  \end{equation*}

\noindent\textsl{Step 6} One choice for $Q$ and $b_0$: 
$Q=\left[\begin{matrix}1 & 1 \\ 0 & 1\end{matrix}\right]$, 
$b_0=\left[\begin{matrix}1\\0\end{matrix}\right]$
\medskip

\noindent\textsl{Step 7}
$A=Q\left[\begin{matrix}m & 0
\\ 0 &n \end{matrix}\right]= \left[\begin{matrix}1 & 1 \\ 0 &
1\end{matrix}\right] \left[\begin{matrix}4 & 0 \\ 0 &
1\end{matrix}\right] =\left[\begin{matrix}4 & 1 \\ 0 & 1
\end{matrix}\right]$, $b=Ab_0=\left[\begin{matrix}4\\
0\end{matrix}\right]$
\medskip

\noindent\textsl{Step 8} Figure~\ref{fig:degfourprendgm} has a
presentation diagram for a NET map whose dynamic portrait is
isomorphic to the initial portrait in Figure~\ref{fig:degfourgraph}.

  \begin{figure}
\centerline{\includegraphics{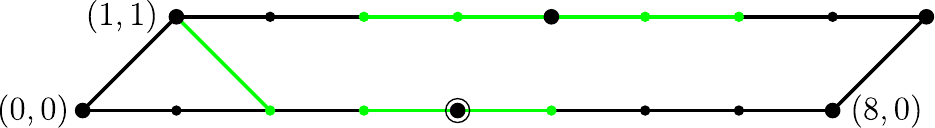}} 
\caption{A presentation diagram for Example~\ref{ex:degfour}} 
\label{fig:degfourprendgm}
  \end{figure}

\end{ex}


\begin{thebibliography}{9}

\bibitem{bd0}
L.~Bartholdi and D.~Dudko. {\em Algorithmic aspects of branched
coverings}. \url{https://arxiv.org/abs/1512.05948}. Submitted.

\bibitem{BN}
L.~Bartholdi and V.~Nekrashevych. {\em Thurston equivalence of topological polynomials}. Acta. Math. {\bf 197}(1), 2006, 1--51.

\bibitem{bekp}
X.~Buff, A.~Epstein, S.~Koch, and K.~Pilgrim, {\em On Thurston's Pullback Map}, in Complex Dynamics--Family and Friends, 561--583, Dierk Schleicher, ed. A K Peters/CRC Press, 2009.

\bibitem{cfpp}
J.~W.~Cannon, W.~J.~Floyd, W.~R.~Parry and K.~M.~Pilgrim, {\em Nearly
Euclidean Thurston maps}, Conform. Geom. Dyn. \textbf{16} (2012),
209--255 (electronic).

\bibitem{fm}
Benson Farb and Dan Margalit, {\em A Primer on Mapping Class Groups}, 
Princeton Univ. Press, Princeton, 2012.

\bibitem{fkklpps}
W.~Floyd, G.~Kelsey, S.~Koch, R.~Lodge, W.~Parry, K.~M.~Pilgrim,
E.~Saenz, {\em Origami, affine maps, and complex dynamics}. \url{https://arxiv.org/abs/1612.06449}. Submitted. 

\bibitem{fpp1}
W.~J.~Floyd, W.~R.~Parry and K.~M.~Pilgrim, {\em Presentations of NET
maps}. \url{https://arxiv.org/abs/1701.00443}. Submitted.

\bibitem{h2}
J.~H.~Hubbard, {\em Teichm\"uller theory, volume 2}. Ithaca, NY: Matrix Editions, 2016. 

\bibitem{K}
G.~Kelsey, {\em Mapping schemes realizable by obstructed topological polynomials}
Conform. Geom. Dyn. 16 (2012), 44-80. 

\bibitem{Ko}
S.~Koch, {\em Teichm\"uller theory and critically finite endomorphisms}, Advances in Mathematics
Vol. 248, 2013. 

\bibitem{kps}
Sarah Koch, Kevin M. Pilgrim and Nikita Selinger, {\em Pullback
invariants of Thurston maps}, Trans. Amer. Math. Soc. \textbf{368} (2016), no. 7, 4621-4655. 

\bibitem{L}
R.~Lodge. {\em Boundary values of the Thurston pullback map}. Conform. Geom. Dyn. 17 (2013), 77-118. 

\bibitem{M}
T.~Miyake, {\em Modular Forms},
Springer-Verlag, Berlin Heidelberg New York, 1989.

\bibitem{N}
V.~Nekrashevych, {\em Self-Similar Groups},
Math. Surveys and Monographs \textbf{117},
Amer. Math. Soc., Providence, 2005.

\bibitem{NET}
The NET map website, \url{www.math.vt.edu/netmaps/}.

\bibitem{NETmap} W.~Parry, {\tt NETmap}, software, available from \url{http://www.math.vt.edu/netmaps}.

\bibitem{pp} 
M. A. Pascali and C. Petronio, {\em Surface branched covers and
geometric 2-orbifolds}, Trans. Amer. Math. Soc. \textbf{361} (2009),
5885--5920. 

\end{thebibliography}
\end{document}